\pdfoutput=1
\documentclass[reqno,11pt,final]{amsart}
\usepackage{glaudo_pkg_en}

\title[Stability for weighted isoperimetric inequalities]{Sharp quantitative stability for \\isoperimetric inequalities with homogeneous weights}

\author[E. Cinti]{E. Cinti}
\address{E.C. --- Alma Mater Studorium Università di Bologna, Dipartimento di Matematica, Piazza di Porta San Donato 5, 40126 Bologna, Italy.}
\email{eleonora.cinti5@unibo.it}

\author[F. Glaudo]{F. Glaudo}
\address{F.G. --- ETH Z\"urich, Mathematics Dept., R\"amistrasse 101, 8092 Z\"urich, Switzerland.}
\email{federico.glaudo@math.ethz.ch}

\author[A. Pratelli]{A. Pratelli}
\address{A.P. --- Universitá di Pisa, Dipartimento di Matematica, Largo B. Pontecorvo 5, 56127 Pisa, Italy.}
\email{aldo.pratelli@dm.unipi.it}

\author[X. Ros-Oton]{X. Ros-Oton}
\address{X.R. --- Universit\"at Z\"urich, Institut f\"ur Mathematik, Winterthurerstrasse 190, 8057 Z\"urich, Switzerland \&
ICREA, Pg. Llu\'is Companys 23, 08010 Barcelona, Spain \&
Universitat de Barcelona, Departament de Matem\`atiques i Inform\`atica, Gran Via de les Corts Catalanes 585, 08007 Barcelona, Spain.}
\email{xavier.ros-oton@math.uzh.ch}

\author[J. Serra]{J. Serra}
\address{J.S. --- ETH Z\"urich, Mathematics Dept., R\"amistrasse 101, 8092 Z\"urich, Switzerland.}
\email{joaquim.serra@math.ethz.ch}



\begin{document} 

\begin{abstract}
We prove the sharp quantitative stability for  a wide class of weighted isoperimetric inequalities.
More precisely, we consider isoperimetric inequalities in convex cones with homogeneous weights. 

Inspired by the proof of such isoperimetric inequalities through the ABP method (see \cite{CRS}), we construct a new convex coupling (i.e., a map that is the gradient of a convex function) between a generic set $E$ and the minimizer of the inequality (as in Gromov's proof of the isoperimetric inequality). 
Even if this map does not come from optimal transport,  and even if there is a weight in the inequality, we adapt the methods of \cite{FMP} and prove that if $E$ is almost optimal for the inequality then it is quantitatively close to a minimizer \emph{up to translations}. 
Then, a delicate analysis is necessary to rule out the possibility of translations.

As a step of our proof, we establish a sharp regularity result for \emph{restricted} convex envelopes of a function that might be of independent interest.
\end{abstract}

\maketitle


\section{Introduction}

\addtocontents{toc}{\protect\setcounter{tocdepth}{1}}
\subsection{Background}
The quantitative stability of functional/geometric inequalities has been an increasingly active field in recent years. The interest lies in understanding whether \emph{almost} minimizers of a certain inequality (i.e., the isoperimetric inequality or the Sobolev inequality) are \emph{quantitatively} close to a minimizer. Let us mention the works \cite{FuscoMaggiPratelli2008,CicaleseLeonardi2012,FMP,FI,BogeleinDuzaarScheven2015,BarchiesiBrancoliniJulin2017,IndreiNurbekyan2015,Neumayer2020,FigalliZhang2020} regarding the stability of the isoperimetric inequality, \cite{BianchiEgnell1991,FuscoMaggiPratelli2007,CianchiFuscoMaggiPratelli2009,nguyen2019,FigalliGlaudo2020,FigalliNeumayer2019} about the stability of Sobolev inequalities and \cite{MaggiPonsiglionePratelli2014,BoroczkyKaroly2010,FigalliMaggiMooney2018,FigalliJerison2017} for the stability of various other inequalities.

On the other hand, many different kinds of weighted isoperimetric inequalities have been established in the literature, let us mention \cite{MorganPratelli2013,Milman2015,BobkovLedouox2009,BettaEtAl1999,BarchiesiBrancoliniJulin2017,Borell75,CianchiFuscoMaggiPratelli2011,RosalesEtAl2008,Chambers2019,CRS}.

The papers that lie at the intersection of the two topics, i.e., quantitative weighted isoperimetric inequalities, are rather rare. Up to our knowledge, only \cite{CianchiFuscoMaggiPratelli2011,BarchiesiBrancoliniJulin2017} prove a quantitative weighted isoperimetric inequality. Both papers consider only the case of the Gaussian weight on $\R^n$.

In this paper we establish for the first time sharp quantitative stability for a wide class of weighted isoperimetric inequalities.
More precisely, we do so for the class of weighted isoperimetric inequalities in convex cones considered in \cite{CRS}. 
The analogous stability result without weights is proven with different methods in \cite{FI}.

\subsection{Results}
Given an open convex cone $\Sigma\subseteq\R^n$ and a weight $w:\Sigma\to\co0\infty$, let us define the weighted volume and perimeter of a set $E\subseteq\Sigma$ with smooth boundary as
\begin{equation*}
    w(E) \defeq \int_{E} w(x)\de x \comma 
    \qquad 
    \Per_w(E) \defeq \int_{\partial E\cap\Sigma}w(x)\de\Haus^{n-1}(x) \fullstop
\end{equation*}
See \cref{sec:notation} for the definition of the perimeter $\Per_w(E)$ for nonsmooth sets. 
Notice that only the boundary of $E$ that is \emph{inside} the cone matters when computing the perimeter $\Per_w(E)$.

Let us recall the weighted isoperimetric inequality in a convex cone.
\begin{theorem}[\cite{CRS}]\label{isop-thm}
Let $\Sigma\subseteq\R^n$ be an open convex cone with vertex at $0$, and let $\alpha>0$.
Let $w$ be a nonnegative continuous function in $\overline\Sigma$ (not constantly $0$) such that
\[w \text{\  is \ } \alpha\text{-homogeneous\footnotemark, \ and \ } w^{1/\alpha} \text{\ is concave in\ } \Sigma.\]
\footnotetext{Here, and everywhere in this paper, we say that a function $w:\Sigma\to\R$ is $\alpha$-homogeneous if $w(tx)=t^\alpha w(x)$ for any $x\in \overline\Sigma$ and any $t>0$.}
Then, for all measurable sets $E\subseteq\Sigma$ with $w(E)<\infty$,
\begin{equation}\label{isop-ineq}
\frac{\Per_w(E)}{w(E)^{\frac{D-1}{D}}}\geq \frac{\Per_w(B_1\cap \Sigma)}{w(B_1\cap \Sigma)^{\frac{D-1}{D}}},
\end{equation}
where $D\defeq n+\alpha$. By homogeneity, $B_1$ can be replaced by $B_r$ for any $r>0$.
\end{theorem}
See \cite[Remark 1.4]{CRS} for a geometric justification of the concavity condition on the weight and \cite[Section 2]{CRS} for a number of examples of admissible weights.
An important example to have in mind is given by monomial weights:
\begin{equation*}
    w(x) = x_1^{A_1}\cdots x_n^{A_n} \quad\textrm{in}\quad \Sigma=\{x_1>0,\dots, x_n>0\},\quad \textrm{with}\ A_i>0 \fullstop
\end{equation*}
The authors do not give a characterization of the sets that achieve the equality in \cref{isop-ineq} (see \cite[2979]{CRS}). 

Our first result is the following characterization of the optimal sets.

\begin{proposition}\label{prop:uniqueness}
    Let $n,\alpha,\Sigma$ and $w$ be as in \cref{isop-thm} and assume\footnote{This assumption is always satisfied up to a rotation.} that $\Sigma=\R^k\times\widetilde\Sigma$, where $0\le k < n$ and $\widetilde\Sigma\subseteq\R^{n-k}$ is an open convex cone \emph{containing no lines}. 
    
    Then, a measurable set $E\subseteq\Sigma$, with $w(E)<\infty$, achieves the equality in \cref{isop-ineq} if and only if $E=B_r(x_0)$ for some $r>0$ and $x_0\in\R^k\times\{0_{\R^{n-k}}\}$.
\end{proposition}

From now on we assume, without loss of generality, that $\Sigma=\R^k\times\widetilde\Sigma$, where $0\le k < n$ and $\widetilde\Sigma\subseteq\R^{n-k}$ is an open convex cone \emph{containing no lines}.
Let us measure the distance between a set $E\subseteq\Sigma$ and the minimizers of the weighted isoperimetric inequality with the quantity
\begin{equation*}
    A_w(E) \defeq
    \inf_{x_0\in\R^k\times\{0_{\R^{n-k}}\}}
    \frac{
    w(E\Delta (B_r(x_0)\cap \Sigma))
    }{
    w(E)
    } \comma
\end{equation*}
where $r>0$ is such that $w(E)=w(B_r\cap \Sigma)$.

We define the weighted isoperimetric deficit of a set $E\subseteq\Sigma$ as
\[\delta_w(E)=\frac{\Per_w(E)}{c_*w(E)^{\frac{D-1}{D}}}-1,\]
where $c_*\defeq\Per_w(B_1\cap \Sigma)/w(B_1\cap \Sigma)^{\frac{D-1}{D}}=Dw(B_1\cap\Sigma)^{\frac1D}$ is the isoperimetric constant that comes from \cref{isop-ineq}, and  $D\defeq n+\alpha$. Notice that the identity $\Per_w(B_1\cap \Sigma)=Dw(B_1\cap \Sigma)$ follows from the homogeneity of the weight $w$.

The main result of the present paper is the following quantitative version of the weighted isoperimetric inequality.

\begin{theorem}\label{thm:main}
    Let $n$, $\alpha$, $\Sigma$ and $w$ be as in \cref{isop-thm} and assume that $\Sigma=\R^k\times\widetilde\Sigma$, where $0\le k<n$ and $\widetilde\Sigma\subseteq\R^{n-k}$ is an open convex cone containing no lines.
    
    Then, for all measurable sets $E\subseteq\Sigma$ with $w(E)<\infty$, it holds
    \begin{equation}\label{eq:quant-ineq}
    A_w(E)\leq C\sqrt{\delta_w(E)},
    \end{equation}
where $C$ is a constant that depends only on $n$, $\alpha$, $\Sigma$ and $w$.
\end{theorem}

Notice that the \emph{stability constant} $C$ contained in \cref{thm:main} cannot depend only on $n$ and $\alpha$.
Indeed, if $n=2$, $\Sigma=\{x_2 > \eps \abs{x_1}\}$ and $w=x_2$, then, as $\eps\to0$, the constant $C$ of the statement must go to infinity (as any ball with center on $\partial\Sigma$ becomes \emph{almost} optimal).

Let us also point out that the exponent $\frac12$ is sharp, as often happens in quantitative stability estimates. We will prove this fact in \cref{rem:sharpness_exponent}.
    
\begin{remark}
    Our proof can be easily adapted (see \cref{thm:anisotropic_coupling}) to recover the stability result of \cite{FMP} for the anisotropic isoperimetric inequality in $\R^n$.
\end{remark}

\subsection{Sketch of the proof}
Our general plan to prove \cref{thm:main} is to make quantitative the proof of the weighted isoperimetric inequality contained in \cite{CRS}. 
However, several new difficulties (both conceptual and technical) arise.
In this description of the proof we will ignore all technical issues; some estimates are stated in a simplified form that is not exactly what we prove (but is morally equivalent). Only in this section\footnote{Later on, we will not consider the cone $\Sigma$ and the weight $w$ as fixed and so a constant that depends on them will not be absorbed by $\lesssim$ (see \cref{sec:notation}).}, the notation $A \lesssim B$ is equivalent to $A\le CB$, where $C$ is a constant that depends on $n,\alpha,\Sigma,w$.

Let us begin by briefly describing the proof of \cref{isop-thm} found in \cite{CRS}.
Given a smooth bounded connected set $E\subseteq\Sigma$ with $w(E)=w(B_1\cap\Sigma)=1$, consider the elliptic problem
\begin{equation}\label{eq:intro_elliptic}
    \begin{cases}
    \div(w\nabla u) = w\, \frac{\Per_w(E)}{w(E)} &\text{in $E$} \\
    \partial_\nu u = 1 &\text{on $\partial E\cap\Sigma$} \\
    \partial_\nu u = 0 &\text{on $\partial E\cap\partial\Sigma$}\fullstop
    \end{cases}
\end{equation}
A contact argument implies that, for any $\xi\in B_1\cap\Sigma$, there is $x\in E$ such that $\nabla u(x) = \xi$ and $\nabla^2 u(x)\ge 0$. Hence, denoting by $E'$ the set $\{x\in E:\ \nabla^2 u\ge 0,\ \nabla u\in B_1\cap\Sigma\}$, the area formula implies
\begin{align*}
    w(B_1\cap\Sigma) 
    &\le 
    \int_{E'}w(\nabla u)\det(\nabla^2 u)
    =
    \int_{E'}\frac{w(\nabla u)}w\det(\nabla^2 u)\, w \\
    &\le
    \int_{E'}\Big(\frac{\tr(\nabla^2 u) + \alpha\big(\frac{w(\nabla u)}w\big)^{\frac1\alpha}}{D}\Big)^D w
    \le
    \int_{E'}\Big(\frac{\lapl u + \frac{\nabla w}{w}\cdot\nabla u}{D}\Big)^D w \\
    &=
    \Big(\frac{\Per_w(E)}{D\,w(E)}\Big)^Dw(E') 
    \le
    \Big(\frac{\Per_w(E)}{D\,w(E)}\Big)^Dw(E)
    \comma
\end{align*}
where we have applied the weighted arithmetic-geometric mean inequality (recall that $D=n+\alpha$), then \cite[Lemma 5.1]{CRS} (which assumes only the concavity of $w^{\frac1\alpha}$) and finally the fact that $u$ satisfies \cref{eq:intro_elliptic}. Notice that the proven inequality is, up to rearrangement of the terms, the weighted isoperimetric inequality (see \cite{CRS} for the details).
What we have just sketched is the ABP method in a nutshell (in the context of isoperimetric inequalities, the method was introduced in \cite{Cabre2000,Cabre2008}; see \cite{Brendle2019} for a recent striking application of the method\footnote{The author proves a sharp isoperimetric inequality for minimal surfaces (up to codimension $2$) embedded in the Euclidean space.
We believe that our methods might be applied to show a quantitative version of Brendle's result.}).
However, it is very hard to exploit directly the function $u$ to obtain a stability result. The main obstructions being that it is impossible to control any derivative of $u$ (as everything depends wildly on $\partial E$) and that the proof \emph{sees} only $E'$ and not the whole set $E$.

Hence we take an appropriate \emph{restricted} convex envelope of $u$. 
Let $\varphi:\R^n\to\R$ be the convex function
\begin{equation}\label{eq:intro_kenvelope}
    \varphi(x) \defeq \sup\big\{a+\xi\cdot x:\ \xi\in \overline{B_1\cap\Sigma},\, a+\xi\cdot y\le u(y)\ \forall y\in E\big\} \fullstop
\end{equation}
Notice that $\varphi$ is simply the supremum of all affine functions with slope in $\overline{B_1\cap\Sigma}$ that are below $u$.
We go on to prove that $\varphi$ is much more well-behaved compared to $u$ itself (mainly because $\varphi$ is convex and controlling the Laplacian of a convex function is sufficient to control the whole Hessian). 
Precisely, we prove that $\varphi$ is $C^{1,1}$ (with bounds independent of the regularity of $\partial E$) and it holds $\nabla \varphi(\overline E) = \overline{B_1\cap\Sigma}$. Moreover, $\varphi$ retains (in a distilled form) the fact that $u$ satisfies \cref{eq:intro_elliptic}:
\begin{equation}\label{eq:intro_tmp432}
    \lapl\varphi + \alpha\left(\frac{w(\nabla\varphi)}{w}\right)^{\frac1\alpha}
        \le \frac{\Per_w(E)}{w(E)}\fullstop
\end{equation}
To understand the meaning of the last inequality, let us remark that, if $\nabla u\in\Sigma$, it holds
\begin{equation*}
    \lapl u + \alpha\left(\frac{w(\nabla u)}{w}\right)^{\frac1\alpha}
    \le 
    \lapl u + \frac{\nabla w}{w}\cdot \nabla u
    = w^{-1}\div(w\nabla u)
    = \frac{\Per_w(E)}{w(E)} \comma
\end{equation*}
where in the first inequality we used the concavity of $w^{\frac1\alpha}$ (see \cref{lem:5.1} below).
It turns out that \cref{eq:intro_tmp432} is enough for our purposes.
Namely, that the properties of the coupling $\varphi$ are still sufficient to prove the weighted isoperimetric inequality (it is sufficient to replace $u$ with $\varphi$ in the proof). Moreover, as simple byproducts of the proof, we obtain the following estimates:
\begin{align}
    &\int_E \abs{\nabla^2\varphi - 1}\, w 
    \lesssim \delta_w(E)^{\frac12} \comma \label{eq:intro_hessian}\\
    &\int_{\partial E\cap\Sigma} (1-\abs{\nabla\varphi})\,w\de\Haus^{n-1} 
    \lesssim \delta_w(E) \comma \label{eq:intro_boundary}\\
    &\int_E \abs*{w(\nabla\varphi)^{\frac1\alpha}-w^{\frac1\alpha}} 
    \lesssim \delta_w(E)^{\frac12} \label{eq:intro_translation}\fullstop
\end{align}
The estimate \cref{eq:intro_hessian} controls the $L^1$-norm of the differential of $\nabla\varphi$, \cref{eq:intro_boundary} implies that $\nabla\varphi(x)$ almost belongs to $\partial B_1\cap \Sigma$ when $x\in\partial E$. It is a bit harder to grasp the point of \cref{eq:intro_translation}, but it will be clear later on.

In \cite{FMP}, the authors have shown that if a set has a small isoperimetric deficit then, up to a small modification, it must enjoy a nontrivial Poincar\'e inequality and a nontrivial trace inequality. Surprisingly, their ideas can be easily adapted to our weighted setting and therefore we can assume that $E$ has nontrivial weighted Poincar\'e and trace inequalities. The adjective \emph{nontrivial} has to be understood as the fact that the constants of the inequalities do not depend on the set $E$ itself and can be bounded a priori. With these considerations, it is not hard to see that \cref{eq:intro_hessian} and \cref{eq:intro_boundary} imply the existence of $x_0\in\R^n$ such that
\begin{align}
    &\int_E \abs{\nabla\varphi(x) - (x-x_0)}\,w(x)\de x 
    \lesssim \delta_w(E)^{\frac12} \comma \label{eq:intro_tmp671}\\
    &\int_{\partial E\cap\Sigma} \abs{\abs{x-x_0}-1}\, w(x)\de \Haus^{n-1}(x) 
    \lesssim \delta_w(E)^{\frac12} \fullstop \label{eq:intro_tmp586}
\end{align}
The estimate \cref{eq:intro_tmp586} tells us that $\partial E\cap\Sigma$ is \emph{almost} contained in the boundary of $B_1(x_0)$, hence it is natural that from \cref{eq:intro_tmp586} we are able to deduce
\begin{equation}\label{eq:intro_tmp324}
    w(E\triangle (B_1(x_0)\cap\Sigma)) \lesssim \delta_w(E)^{\frac12} \fullstop
\end{equation}
Notice that a similar deduction is present also in \cite{FMP} and \cite{FI}, but their method to prove it does not work in our setting.

It remains to show that in \cref{eq:intro_tmp324} we can choose $x_0=(z, 0_{\R^{n-k}})$ (recall that both the cone and the weight are invariant on the first $k$ coordinates). 
This final step is far from being straightforward. 
Indeed, the interplay between the boundary of the cone and the weight makes it cumbersome to rule out that $E$ is close to a \emph{translated} ball, i.e., a ball $B_r(x_0)$ with $x_0$ not lying in $\R^k\times\{0_{\R^{n-k}}\}$.

Without loss of generality we can assume that $x_0=(0_{\R^k},\widetilde x_0)$. We will show that $\abs{\widetilde x_0}\lesssim \delta_w(E)^{\frac12}$, which is sufficient to conclude the proof.
Instead of giving the details of our strategy, we describe it in three different settings in order to show all the ideas without being lost in the technicalities. In all the three cases it holds $k=0,\,n=2$ (i.e., no lines are contained in $\Sigma$). We will denote the coordinates with $(x, y)$.

The first method works if $x_0$ belongs to $\overline\Sigma$ or $-\overline\Sigma$ and $w$ is constant along the direction of $x_0$.
The second method works if $x_0$ is not aligned with a constant direction of $w$.
The third method works if $x_0$ does not belong to $\Sigma$ nor $-\Sigma$ and $w$ is constant along the direction of $x_0$.
Since we have (morally) covered all possible cases, the proof is concluded.
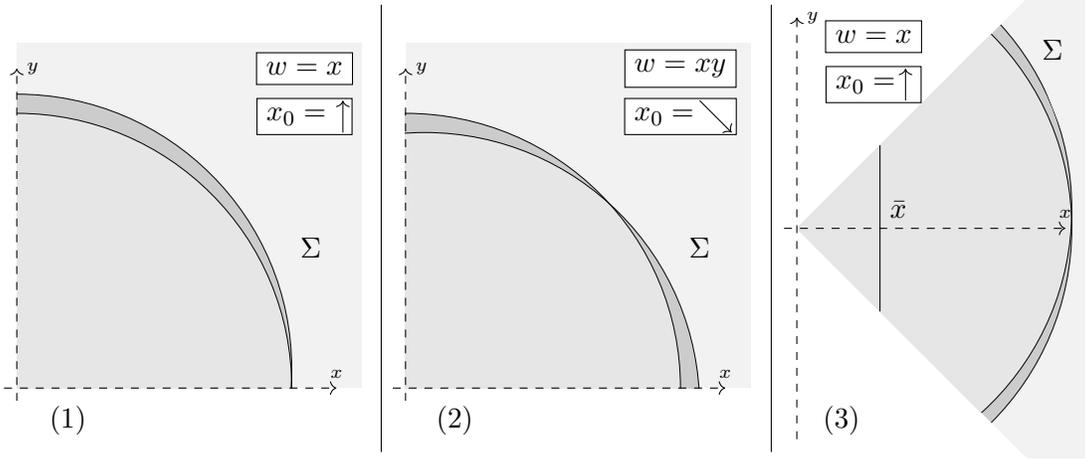
\begin{figure}[htb]
    \centering

\begin{tikzpicture}[scale=0.85]


\begin{scope}[xshift=0]
\fill[fill=white!95!black] (0,0) rectangle (5.4,5.4);

\begin{scope}
    \clip (0,0) rectangle (5,5);
    \clip (0,0) circle (4.3);
    \fill[white!80!black] (0,0) circle(4.3);
    \fill[white!90!black] (0,0.3) circle(4.3);
\end{scope}

\begin{scope}
    \clip (0,0) rectangle (5,5);
    \clip (0,0.3) circle (4.3);
    \fill[white!80!black] (0,0.3) circle(4.3);
    \fill[white!90!black] (0,0) circle(4.3);
\end{scope}

\begin{scope}
    \clip (0,0) rectangle (5,5);
    \draw[ultra thin] (0,0) circle(4.3);
    \draw[ultra thin] (0,0.3) circle(4.3);
\end{scope}

\node[draw,fill=white] at (4.5, 5) {$w=x$};
\node[draw,fill=white] at (4.5,4.25) {$x_0=\phantom{.}$};
\draw[->] (5.1, 4)--(5.1,4.5);

\node at (4.6, 2.2) {$\Sigma$};

\draw[->,dashed] (-0.2,0)--(5,0) node[above]{\tiny{$x$}};
\draw[->,dashed] (0,-0.2)--(0,5) node[right]{\tiny{$y$}};
\end{scope}


\begin{scope}[xshift=6.08cm]
\fill[fill=white!95!black] (0,0) rectangle (5.4,5.4);

\begin{scope}
    \clip (0,0) rectangle (5,5);
    \clip (0,0) circle (4.3);
    \fill[white!80!black] (0,0) circle(4.3);
    \fill[white!90!black] (0.3,-0.3) circle(4.3);
\end{scope}

\begin{scope}
    \clip (0,0) rectangle (5,5);
    \clip (0.3,-0.3) circle (4.3);
    \fill[white!80!black] (0.3,-0.3) circle(4.3);
    \fill[white!90!black] (0,0) circle(4.3);
\end{scope}

\begin{scope}
    \clip (0,0) rectangle (5,5);
    \draw[ultra thin] (0,0) circle(4.3);
    \draw[ultra thin] (0.3,-0.3) circle(4.3);
\end{scope}

\node[draw,fill=white] at (4.3, 5) {$w=xy$};
\node[draw,fill=white] at (4.3,4.25) {$x_0=\phantom{x.}$};
\draw[->] (4.6, 4.5)--(5.1,4);

\node at (4.6, 2.2) {$\Sigma$};

\draw[->,dashed] (-0.2,0)--(5,0) node[above]{\tiny{$x$}};
\draw[->,dashed] (0,-0.2)--(0,5) node[right]{\tiny{$y$}};

\end{scope}


\begin{scope}[xshift=12.2cm, yshift=2.5cm]
\begin{scope}
    \clip (0, -3.6) rectangle (4.5, 3.6);
    \fill[fill=white!95!black] (0,0) -- (5, 5) -- (5, -5) -- cycle;
\end{scope}

\begin{scope}
    \clip (0,0) -- (3.4, 3.4) -- (4.35, 1) -- (4.35, -1) -- (3.4, -3.4) -- cycle;
    \clip (0,0) circle (4.3);
    \fill[white!80!black] (0,0) circle(4.3);
    \fill[white!90!black] (0,0.3) circle(4.3);
\end{scope}

\begin{scope}
    \clip (0,0) -- (3.4, 3.4) -- (4.35, 1) -- (4.35, -1) -- (3.4, -3.4) -- cycle;
    \clip (0,0.3) circle (4.3);
    \fill[white!80!black] (0,0.3) circle(4.3);
    \fill[white!90!black] (0,0) circle(4.3);
\end{scope}

\begin{scope}
    \clip (0,0) -- (3.4, 3.4) -- (4.35, 1) -- (4.35, -1) -- (3.4, -3.4) -- cycle;
    \draw[ultra thin] (0,0) circle(4.3);
    \draw[ultra thin] (0,0.3) circle(4.3);
\end{scope}

\node[draw,fill=white] at (1.2, 3) {$w=x$};
\node[draw,fill=white] at (1.2,2.25) {$x_0=\phantom{.}$};
\draw[->] (1.7, 2.0)--(1.7,2.5);

\node at (4, 2.8) {$\Sigma$};

\draw[->,dashed] (-0.2,0)--(4.2,0) node[above]{\tiny{$x$}};
\draw[->,dashed] (0,-3.3)--(0,3.3) node[right]{\tiny{$y$}};

\draw (1.3,-1.3) -- (1.3,1.3);
\node[above right] at (1.3, 0) {$\bar x$};
\end{scope}

\node at (0.8,-0.5) {(1)};
\node at (6.85,-0.5) {(2)};
\node at (12.9,-0.5) {(3)};
\draw (5.7,-1) -- (5.7,6);
\draw (11.8,-1) -- (11.8,6);
\end{tikzpicture}

    \caption{Visual description of the three fundamental cases that have to be handled to prove $\protect\abs{\widetilde x_0}\lesssim \delta_w(E)^{\frac12}$.}
  \end{figure}
\begin{enumerate}[ref=(\arabic*)]
    \item \label{it:intro_ball_growth} Let $\Sigma\defeq \{x>0,\, y>0\}$, $w\defeq x$ and assume that $x_0=(0, t)$ for some $t\in\R$. It holds
    \begin{equation*}
        \abs*{w(B_1(x_0)\cap\Sigma)-w(B_1\cap\Sigma)} \gtrsim \abs{t} \comma
    \end{equation*}
    that, together with \cref{eq:intro_tmp324}, implies $t\lesssim \delta_w(E)^{\frac12}$ as desired (recall that $w(E)=w(B_1\cap\Sigma)$).
    \item \label{it:intro_nonconstant} Let $\Sigma\defeq \{x >0,\, y > 0\}$, $w\defeq xy$ and assume that $x_0=(t, -t)$ for some $t\in\R$. Joining \cref{eq:intro_translation} and \cref{eq:intro_tmp671}, we can prove
    \begin{equation*}
        \int_E \abs{w^{\frac1\alpha}(x-x_0)-w^{\frac1\alpha}} 
        \lesssim \delta_w(E)^{\frac12} \fullstop
    \end{equation*}
    Since the weight $w$ is nonconstant in the direction $(1,-1)$, it holds
    \begin{equation*}
        \int_E \abs{w^{\frac1\alpha}(x-x_0)-w^{\frac1\alpha}} \gtrsim t \fullstop
    \end{equation*}
    The last two inequalities imply $t\lesssim \delta_w(E)^{\frac12}$.
    \item \label{it:intro_constant_hard} Let $\Sigma\defeq \{x>\abs{y}\}$, $w\defeq x$ and assume that $x_0=(0, t)$ for some $t\in\R$. This is the hardest situation: the weight is constant along $x_0$ (hence \cref{eq:intro_translation} is useless) and the value of $w(B_1(x_0)\cap\Sigma)$ can be very close to $w(B_1\cap\Sigma)$ (so we cannot use their difference to control $t$ as we did in the first case). Thanks to \cref{eq:intro_hessian} and \cref{eq:intro_tmp671}, applying the fundamental theorem of calculus we can find $\frac14\le \overline x\le \frac12$ such that
    \begin{equation*}
        \abs{\nabla\varphi(\overline x, y) - ((\overline x, y)-x_0)} \lesssim \delta_w(E)^{\frac12}
    \end{equation*}
    for any $y\in\R$ such that $(\overline x, y)\in\Sigma$ (restricting our attention to a $1$-dimensional segment we have improved an $L^1$ estimate to an $L^\infty$ estimate and this is fundamental). 
    Since $\nabla\varphi\in\overline{B_1\cap\Sigma}$, we deduce
    \begin{equation*}
        \text{dist}(\Sigma, (\overline x, y)-x_0) \lesssim \delta_w(E)^{\frac12} \fullstop
    \end{equation*}
    Choosing $y=\overline x$ or $y=-\overline x$ in the latter estimate (depending on the sign of $t$), 
    we readily deduce $\abs{t}\lesssim \delta_w(E)^{\frac12}$.
\end{enumerate}

\subsection{Comparison between our coupling and the optimal transport map}
Let $E\subseteq\R^n$ be a bounded connected set with smooth boundary.
If we set $\Sigma=\R^n$ and $w=1$, the sketch above\footnote{It is sufficient to repeat the sketch, ignoring the weight and the cone, until \cref{eq:intro_tmp432}. Alternatively, see \cref{thm:anisotropic_coupling}.} provides a \emph{convex} function $\varphi:E\to\R$ such that $\nabla\varphi(E)=B_1$ (up to negligible sets) and 
\begin{equation*}
    \tr(\nabla^2\varphi) \le \frac{\Per(E)}{\abs{E}} \fullstop
\end{equation*}
On the other hand, let $\nabla\psi:E\to B_1$ be the optimal transport map between the two probability measures $\abs{E}^{-1}\restricts{\Leb^n}{E}$ and $\abs{B_1}^{-1}\restricts{\Leb^n}{B_1}$ with respect to the quadratic cost (see \cite{FMP} or \cite[23]{Villani2009} for the missing details). The function $\psi:E\to\R$ is \emph{convex} (\cite{Brenien1991}) and, by definition of transport map, it holds $\nabla\psi(E) = B_1$ (up to negligible sets) and
\begin{equation*}
    \det(\nabla^2\psi) = \frac{\abs{B_1}}{\abs{E}} \fullstop
\end{equation*}

Summing up, both $\varphi,\psi:E\to\R$ are convex functions such that $\nabla\varphi(E)=\nabla\psi(E) = B_1$ (up to negligible sets) and they both satisfy a condition on the Hessian.
Notice also that both $\varphi$ and $\psi$ encode some nontrivial global information about the set $E$, indeed both yield a one-line proof of the isoperimetric inequality:
\begin{align}
    &\abs{B_1}=\abs{\nabla\varphi(E)}
    \le \int_E \det(\nabla^2\varphi) 
    \le \int_E \left(\frac{\tr(\nabla^2\varphi)}{n}\right)^n 
    \le \frac{\Per(E)^n}{n^n\abs{E}^{n-1}} \comma \label{eq:proof_iso_abp}\\
    &n\abs{E}\left(\frac{\abs{B_1}}{\abs{E}}\right)^{\frac1n}
    = n\int_E \det(\nabla^2\psi)^{\frac1n}
    \le \int_E \div(\nabla\psi)
    = \int_{\partial E}\nabla\psi\cdot \nu_{\partial E} \de\Haus^{n-1}
    \le \Per(E) \fullstop \nonumber
\end{align}
Given their many similarities, it is natural to wonder whether the two functions $\varphi, \psi$ are two instances of the same phenomenon. 
We believe it would be very interesting to find a unifying framework that allows to treat the two functions (and perhaps other functions with in-between conditions on the Hessian) together.

\subsection{Acknowledgements}

F.G. and J.S. have received funding from the European Research Council under the Grant Agreement No. 721675 ``Regularity and Stability in Partial Differential Equations (RSPDE)''.

X.R. has received funding from the European Research Council under the Grant Agreement No. 801867 ``Regularity and singularities in elliptic PDE (EllipticPDE)''.

X.R. and J. S. were supported by grant MTM2017-84214-C2-1-P.

\subsection{Organization of the paper}
While describing the content of the various sections of the work we refer to the sketch of the proof given above.

After a section of notation and preliminaries, in \cref{sec:kenvelope} we study the $K$-envelope of a function (where $K$ is a compact convex set) obtaining a precise $C^{1,1}$-regularity result that might be of independent interest (and that we will use later on in the proof). When $K=\overline{B_1\cap\Sigma}$, the $K$-envelope of a generic function $u$ boils down to \cref{eq:intro_kenvelope} (thus the coupling $\varphi$ is a $\overline{B_1\cap\Sigma}$-envelope).
The construction of the coupling $\varphi$ and the proof of its properties are contained in \cref{sec:coupling}. 
We construct the coupling also in the case of the anisotropic (unweighted) perimeter, as it could be of independent interest.
The strategies adopted to deduce $\abs{\widetilde x_0}\le C \delta_w(E)^{\frac12}$ are implemented in \cref{sec:ruling_out_translations}.
The implication \cref{eq:intro_tmp586}$\implies$\cref{eq:intro_tmp324} is proven in \cref{sec:spherical_boundary_implies_ball}.
In \cref{sec:FMP}, we adapt \cite[Section 3]{FMP} to the weighted setting. Namely, we show that if a set has a small weighted isoperimetric deficit, then it enjoys nontrivial trace and Poincar\'e inequalities. 
The proofs of \cref{prop:uniqueness} and \cref{thm:main} are contained in \cref{sec:main_proof}.

This work has three appendices. The first one contains a quantitative version of the weighted inequality of arithmetic and geometric means; the second one is a collection of simple facts regarding $1$-homogeneous concave functions in a cone.
In the third and final appendix we prove that, in $\R^2$, an indecomposable set (see \cref{def:indecomposable}) can be approximated with \emph{connected} open sets.

\section{Notation}\label{sec:notation}
Since the statement of \cref{thm:main} is invariant under rescaling of the weight, we can assume that $w(B_1\cap\Sigma)=1$. 
Notice that, under this additional constraint, the isoperimetric constant is simply $c_*=D$ (recall that $D\defeq n+\alpha$).

Any constant that depends only on $n, \alpha$ is considered universal and can be hidden in the notation $\lesssim$. Precisely, the notation $A\lesssim B$ is equivalent to $A\le C B$ for a suitable constant $C=C(n,\alpha)$ that depends only on $n$ and $\alpha$. 
On the other hand, we will write explicitly constants that depend on the cone $\Sigma$ and the weight $w$.

We denote with $B_r(x_0)$ the open ball with center $x_0\in\R^n$ and radius $r>0$; when the center is the origin ($x_0=0_{\R^n}$) we simply write $B_r$. The $(n-1)$-dimensional sphere is denoted by $\partial B_1$ or $\S^{n-1}$.
The $n$-dimensional Lebesgue measure is denoted by $\Leb^n$, the $(n-1)$-dimensional Hausdorff measure is denoted by $\Haus^{n-1}$. The Lebesgue measure of a set $E\subseteq\R^n$ is denoted by $\abs{E}$. The identity matrix (whose size will always be $n\times n$) is denoted by $\id$.
The convex-hull of a set $E$, that is the smallest closed convex set that contains $E$, is denoted by $\conv(E)$.

\subsection{Assumptions}
Almost all the statements of this work require the same assumptions on the cone and the weight. For notational simplicity we state them here and reference them instead of repeating them in every statement.
We assume that $n,\alpha,\Sigma,w$ satisfy
\begin{equation}\label{eq:assumptions}\begin{cases}
    n \in \N \text{ and } \alpha \in \oo0\infty ;\\
    \Sigma \subseteq\R^n \text{ is an open convex cone with vertex at $0$;} \\
    \Sigma = \R^k\times\widetilde\Sigma \text{ where $0\le k < n$ and $\widetilde\Sigma\subseteq\R^{n-k}$ is an open convex cone \emph{containing no lines};} \\
    w:\overline\Sigma\to\co0\infty \text{ is a continuous nonnegative weight such that $w$ is $\alpha$-homogeneous, }\\ \text{$w^{\frac1\alpha}$ is concave, and $w(B_1\cap\Sigma)=1$.}
\end{cases}\end{equation}
Notice that $w$ may be $0$ on $\partial\Sigma$, but, since $w^{\frac1\alpha}$ is concave, it is strictly positive inside $\Sigma$.

A simple but useful result from \cite{CRS} is the following.

\begin{lemma}[\cite{CRS}]\label{lem:5.1}
Let $w$ be a positive homogeneous function of degree $\alpha>0$ in an open cone $\Sigma\subseteq\R^n$.
Then, the following conditions are equivalent:
\begin{itemize}
\item The function $w^{1/\alpha}$ is concave in $\Sigma$.

\item For each $x, z\in\Sigma$, the following inequality holds:
\[\alpha\left(\frac{w(z)}{w(x)}\right)^{1/\alpha}\leq \frac{\nabla w(x)\cdot z}{w(x)}.\]

\end{itemize}
\end{lemma}

We will use such result several times throughout the paper.

\subsection{Set of finite (weighted) perimeter and functions of bounded variation}
Let us recall some basic facts about sets of finite perimeter. 
All the definitions and results we are going to state can be found in the monograph \cite{maggi2012}.

A measurable set $E\subseteq\R^n$ is a set of finite perimeter if its perimeter
\begin{equation*}
    \Per(E) \defeq \sup \left\{\int_E \div(X(x))\de x:\ 
    X\in C^{\infty}_c(\R^n,\R^n)\comma\, \norm{X}_{\infty} \le 1\right\}
\end{equation*}
is finite. 
There is also a \emph{localized} version of the perimeter; given an open set $\Omega\subseteq\R^n$, the relative perimeter of $E$ inside $\Omega$ is
\begin{equation*}
    \Per(E,\Omega) \defeq \sup \left\{\int_E \div(X(x))\de x:\ 
    X\in C^{\infty}_c(\Omega,\R^n)\comma\, \norm{X}_{\infty} \le 1\right\} \fullstop
\end{equation*}
A set of (locally) finite perimeter admits a measure-theoretic notion of boundary (the \emph{reduced boundary}), which is a $(n-1)$-rectifiable set that we denote by $\bou E$, and for each point $x\in\bou E$ an \emph{outer normal vector} $\nu_{\bou E}(x)\in\S^{N-1}$ is defined. For any $X\in C^\infty_c(\R^n, \R^n)$, it holds
\[
\int_E\div(X(x))\de x =\int_{\bou E} X \cdot \nu_{\bou E}\de\H^{n-1}\,.
\]
In addition, $\Per(E)=\Haus^{n-1}(\bou E)$.
The reduced boundary is, up to $\Haus^{n-1}$-negligible sets, the set of points in $\R^n$ where $E$ has density $\frac12$.
We denote with $E^{(1)}$ the set of points in $\R^n$ where $E$ has density $1$.

Let us now define the \emph{weighted} perimeter.
Given a measurable set $E\subseteq\Sigma$, its $w$-perimeter in $\Sigma$ is defined as
\begin{equation*}
    \Per_w(E) \defeq \sup \left\{\int_E \div(X(x)w(x))\de x:\ 
    X\in C^{\infty}_c(\Sigma,\R^n)\comma\, \norm{X}_{\infty} \le 1\right\} \fullstop
\end{equation*}
It is not hard to prove that if $\Per_w(E)<\infty$, then for any $\Omega\compactsubset \Sigma$ it holds $\Per(E, \Omega)<\infty$. In particular $\bou E$ is well-defined whenever $\Per_w(E)<\infty$ and it holds
\begin{equation*}
    \Per_w(E) = \int_{\bou E\cap\Sigma} w(x)\de\Haus^{n-1}(x) \fullstop
\end{equation*}
Let us define $\Haus_w^{n-1}\defeq w\restricts{\Haus}{\Sigma}^{n-1}$, so that $\Per_w(E) = \Haus_w^{n-1}(\bou E)$.

Let us now move our attention to functions of bounded variation.
All the definitions and results we are going to state can be found in the monograph \cite{ambrosio-fusco-pallara}.

A measurable function $f:\R^n\to\R$ is of bounded variation if its distributional gradient is a vector-valued measure, that we denote with\footnote{The distributional gradient of a function of bounded variation is usually denoted with $D f$, and $\nabla f$ is used to identify the absolutely continuous part of $D f$. We use the notation $\diff f$ for the distributional gradient to avoid confusion, indeed the letter $D$ is the \emph{effective dimension} $D=n+\alpha$.} $\diff f$. The set of functions of bounded variation is denoted by $BV(\R^n)$.

Given $f\in BV(\R^n)$ and $X\in C^{\infty}_c(\R^n,\R^n)$, it holds
\begin{equation}\label{eq:intro_div_theorem}
    \int_{\R^n} f(x)\div(X(x))\de x = 
    \int_{\R^n} X(x) \de \diff f(x) \fullstop
\end{equation}
Notice that a measurable set $E\subseteq\R^n$ is a set of finite perimeter if and only if its characteristic function $\chi_E$ has bounded variation.
Moreover,  the following relation between the reduced boundary of $E$ and the distributional gradient of the characteristic function,
\begin{equation*}
    \diff \chi_E = -\nu_{\bou E}\, \restricts{\Haus}{\bou E}^{n-1} \comma
\end{equation*}
holds.

\section{Regularity of the \texorpdfstring{$K$}{K}-envelope}\label{sec:kenvelope}
Let us start with the definition of the main character of this section: the $K$-envelope of a function.

\begin{definition}
Let $K\subseteq\R^n$ be a compact, convex set, $\Omega\subseteq\R^n$ be a bounded open set, and $u\in C^0(\overline\Omega)\cap C^2(\Omega)$.
   We define the \emph{$K$-envelope of $u$} as the function $\overline u^K:\R^n\to\R$ given by
    \begin{equation*}
        \overline u^K(x) \defeq \sup\left\{a + \xi\cdot x:\ \xi\in K,\ a+\xi\cdot y\le u(y) \ \forall y\in\overline\Omega\right\} \fullstop
    \end{equation*} 
\end{definition}
\begin{remark}
    The $K$-envelope is, by definition, the supremum of all affine functions with slope in $K$ that are below $u$. Notice that $\overline u^K(x)<\infty$ because $K$ is compact.
\end{remark}
\begin{remark}
    For our purposes, only the case $K=\overline{B_1\cap\Sigma}$ is important. Nonetheless, since it is not much easier to handle only that case and because the results in this section might be of independent interest, we decided to drop the assumption $K=\overline{B_1\cap\Sigma}$ and work with a generic compact convex set.
\end{remark}

The goal of this section is to obtain some precise $C^{1,1}$ bounds on $\overline u^K$. We are interested in showing that the Hessian of $\overline u^K$ is controlled (as a symmetric matrix) by a suitable combination of Hessians (in different points) of the original function $u$.
Similar results are well-known for the classical convex envelope (see for instance \cite{FigalliDePhilippis2015} and the references therein). Nonetheless, we could not find any work on the $K$-envelope. 

The fundamental difficulty arising when considering the $K$-envelope (compared to the convex envelope) is that at many points it holds $\nabla \overline u^K \in \partial K$ and there the standard approaches fail.
This shortcoming can be solved neatly introducing the notion of normal cone.
Let us define the normal cone and the subdifferential (see \cite{rockafellar1970}) and obtain the first basic results about the $K$-envelope.

\begin{definition}[Normal cone]
    Given a compact, convex set $K\subseteq\R^n$, for any $\xi\in K$, the normal cone $N(\xi, K)$ of $K$ at $\xi$ is defined as
    \begin{equation*}
        N(\xi, K) \defeq \left\{v\in\R^n: \,v\cdot(\xi'-\xi) \le 0 \ \textrm{ for all }\ \xi'\in K\right\} \fullstop
    \end{equation*}
    Notice that the normal cone is most interesting for boundary points $\xi\in \partial K$; in the interior one simply has $N(\xi,K)=\{0\}$.
\end{definition}
\begin{definition}[Subdifferential]
    Given a convex function $\varphi:\R^n\to\R$, its subdifferential $\partial \varphi(x)$ at the point $x\in\R^n$ is defined as
    \begin{equation*}
        \partial \varphi(x)\defeq \left\{\xi\in\R^n:\ \varphi(y)\ge \varphi(x) + \xi\cdot(y-x) \ \forall y\in\R^n\right\} \fullstop
    \end{equation*}
\end{definition}

\begin{lemma}\label{lem:subdifferential_cap_K}
Let $K\subseteq\R^n$ be a compact, convex set, $\Omega\subseteq\R^n$ be a bounded open set, and $u\in C^0(\overline\Omega)\cap C^2(\Omega)$.
    The function $\overline u^K:\R^n\to\R$ is convex and at any point $x\in\R^n$ it holds
    $\partial \overline u^K(x)\cap K\not=\emptyset$.
\end{lemma}
\begin{proof}
    The convexity follows directly from the definition. 
    For the second part of the statement, fix $(a_i)_{i\in\N}\subseteq\R$ and $(\xi_i)_{i\in\N}\subseteq K$ such that $\overline u^K(x) = \lim_{i\to\infty} a_i+\xi_i\cdot x$ and $a_i+\xi_i\cdot y \le u(y)$ for any $y\in\overline\Omega$. 
    Since $K$ is compact we can assume that $\xi_i\to\xi\in K$ and as a consequence it must hold $a_i\to a\in\R$. 
    Hence $\overline u^K(x) = a + \xi\cdot x$ and $a+\xi\cdot y\le u(y)$ for any $y\in\overline\Omega$;
    in particular this implies that $\overline u^K(y)\ge a + \xi\cdot y$ for any $y\in\R^n$ and we deduce $\xi\in\partial \overline u^K(x)$.
\end{proof}

\begin{lemma}\label{lem:convex_duality101}
    Let $A, B\subseteq\R^n$ be two nonempty closed convex sets such that for any $b\in B$ there is $a\in A$ such that $a\cdot b\le 0$. Then there is $\overline a\in A$ such that $\overline a\cdot b\le 0$ for any $b\in B$.
\end{lemma}
\begin{proof}
    Given a subset $S\subseteq\R^n$, let $S^\circ$ be its polar cone, that is
    \begin{equation*}
        S^\circ\defeq \{x\in\R^n:\ x\cdot s \le 0 \text{ for any $s\in S$}\} \fullstop
    \end{equation*}
    We want to prove that $A\cap B^{\circ}\not=\emptyset$.
    Let us assume by contradiction that $A\cap B^{\circ}=\emptyset$. Then, since $A$ is a closed convex set and $B^{\circ}$ is a closed convex cone, we can find $v\in B^{\circ\circ}$ such that $a\cdot v > 0$ for any $a\in A$. Since $B^{\circ\circ}$ is the  cone $\{\lambda b:\, b\in B, \lambda\ge 0\}$ generated by $B$ (see \cite[Theorem 14.1]{rockafellar1970}), up to rescaling we can assume that $v\in B$. Thus we have reached a contradiction as we are assuming the existence of $a\in A$ such that $a\cdot v\le 0$.
\end{proof}

\begin{lemma}\label{lem:preparation_convex_K}
Let $K\subseteq\R^n$ be a compact, convex set, $\Omega\subseteq\R^n$ be a bounded open set, and $u\in C^0(\overline\Omega)\cap C^2(\Omega)$.
    Given $\xi\in K$, let $S_\xi\defeq \argmin_{x\in\overline\Omega} \{u(x)-\xi\cdot x\}$. Then we have:
    \begin{enumerate}[label=(\alph*)]
        \item For any $x_\xi\in S_\xi$, it holds $\overline u^K(x_\xi)=u(x_\xi)$ and $\xi\in\partial\overline u^K(x_\xi)$. \label{it:contact_set}
        \item Given $x\in\R^n$, if $\xi\in\partial\overline u^K(x)$ then there is $v\in N(\xi, K)$ such that $x-v\in\conv(S_\xi)$. \label{it:tangent_set}
    \end{enumerate}
\end{lemma}
\begin{proof}
    Let us start proving \cref{it:contact_set}. 
    Since $x_\xi$ minimizes $u(x)-\xi\cdot x$, for any $y\in\overline\Omega$ it holds
    \begin{equation*}
        u(y)\ge u(x_\xi) + \xi\cdot(y-x_\xi)
    \end{equation*}
    and thus, by definition of $K$-envelope, we deduce
    \begin{equation}\label{eq:local3}
        u(y)\ge \overline u^K(y) \ge u(x_\xi) + \xi\cdot(y-x_\xi) \fullstop
    \end{equation}
    Setting $y=x_\xi$, \cref{eq:local3} implies $\overline u^K(x_\xi)=u(x_\xi)$, thus $\xi\in\partial\overline u^K(x_\xi)$.
    
    Let us now move to the proof of \cref{it:tangent_set}.
    Fix $\xi'\in K$ and, for any $0<\eps<1$, choose $y_{\xi',\eps}\in S_{\xi+\eps(\xi'-\xi)}$. 
    Thanks to \cref{it:contact_set}, we know that
    $\xi+\eps\cdot(\xi'-\xi)\in\partial\overline u^K(y_{\xi',\eps})$ and, since the subdifferential is a monotone operator (see \cite[§24]{rockafellar1970}), this implies
    \begin{equation}\label{eq:local1}
        (x-y_{\xi',\eps})\cdot(\xi'-\xi) \le 0 \fullstop
    \end{equation}
    Thanks to the compactness of $\overline\Omega$, up to subsequence, it holds $y_{\xi',\eps}\to y_{\xi'}\in S_\xi$ as $\eps\to 0$ (we are using that $\xi+\eps(\xi'-\xi)\to\xi$ as $\eps\to0$). Thus, passing to the limit in~\cref{eq:local1}, we deduce
    \begin{equation}\label{eq:local2}
        (x- y_{\xi'})\cdot(\xi'-\xi) \le 0 \fullstop
    \end{equation}
    We have shown that for any $\xi'\in K$ there exists $ y_{\xi'}\in S_\xi$ such that \cref{eq:local2} holds. 
    The conclusion follows from \cref{lem:convex_duality101} with $A=x-\conv(S_\xi)$ and $B=K-\xi$.
\end{proof}

We now have all the tools to prove the central result of this section. The idea is the following. Fix $x\in\R^n$ and consider a hyperplane touching $\overline u^K$ from below at $x$. This hyperplane touches $u$ from below at, say, $x_1, \dots, x_m$. We prove that (in a rather strong sense) $\overline u^K$ admits an Hessian at $x$ and this Hessian is controlled by a convex combination of the Hessian of $u$ at $x_1,\dots, x_m$. More precisely, $\nabla^2\overline u^K(x)$ belongs to $H(x, \nabla\overline u^K(x), K)$, which is defined as follows:
\begin{definition}
Let $K\subseteq\R^n$ be a compact, convex set, $\Omega\subseteq\R^n$ be a bounded open set, and $u\in C^0(\overline\Omega)\cap C^2(\Omega)$.
    Given $x\in\R^n$ and $\xi\in K$, let us define the family of matrices (the definition of $S_\xi$ is contained in the statement of \cref{lem:preparation_convex_K})
    \begin{equation*}
        H(x, \xi, K) \defeq\left\{
        \sum_{i=1}^m \lambda_i \nabla^2 u(s_i):\ \ 
        \begin{aligned}
        & 1\le m \le n+1 \\
        & \lambda_i\ge 0,\ \sum \lambda_i = 1 \\
        & s_i\in S_\xi \cap \Omega \\
        & x-\sum \lambda_i s_i \in N(\xi, K)
        \end{aligned}
        \right\} \fullstop
    \end{equation*}
    Notice that $H(x,\xi,K)$ is convex by definition and it is also closed if $S_\xi\subseteq\Omega$ (since $S_\xi$ is closed by definition).
\end{definition}

\begin{proposition}\label{prop:convex_envelope_regularity}
Let $K\subseteq\R^n$ be a compact, convex set, $\Omega\subseteq\R^n$ be a bounded open set, and $u\in C^0(\overline\Omega)\cap C^2(\Omega)$.
    Assume that for any $\xi\in K$ it holds $S_\xi\subseteq \Omega$ (where $S_\xi$ is defined as in \cref{lem:preparation_convex_K}).
    Then $\overline u^K:\R^n\to\R$ is a $C^{1,1}$ convex function such that 
    \[\nabla \overline u^K(\R^n) = \nabla \overline u^K(\Omega) = K.\] 
    Moreover, for any $x\in\R^n$ and any $H_x\in H(x, \nabla \overline u^K(x), K) \not=\emptyset$, it holds $\nabla^2\overline u^K(x)\le H_x$, in the sense that, for $x'$ converging to $x$, 
    \begin{equation}\label{eq:convex_envelope_fundamental}
        \overline u^K(x')\le \overline u^K(x) + \nabla \overline u^K(x)\cdot (x'-x) + \frac12 (x'-x)H_x(x'-x) + \smallo(\abs{x'-x}^2) \fullstop
    \end{equation}
\end{proposition}
\begin{proof}
    Fix $x\in\R^n$ and $\xi\in\partial\overline u^K(x)\cap K$ (the intersection is nonempty by \cref{lem:subdifferential_cap_K}). Thanks to \cref{lem:preparation_convex_K} (and Carath\'eodory's theorem\footnote{Recall that Carath\'eodory's theorem \cite{Caratheodory} states that if a point $x_\circ\in \R^n$ lies in the convex hull of a set $S$, then $x_\circ$ can be written as the convex combination of at most $n+1$ points in $S$.} for convex hulls) we can find $v\in N(\xi, K)$, $s_1,\dots, s_m\in S_\xi\subseteq\Omega$, with $1\le m\le n+1$, and $\lambda_i\geq0$, $\lambda_1+\cdots+\lambda_m = 1$, such that
    \begin{equation*}
        x-v = \lambda_1 s_1 + \cdots + \lambda_m s_m\fullstop
    \end{equation*}
    In particular, $H(x, \nabla \overline u^K(x), K)$ is nonempty. 
    Now, let us fix any such $v$, $(s_i)_{i=1,\dots, m}$ and $(\lambda_i)_{i=1,\dots, m}$.
    
    Take any $x'\in\R^n$ close enough to $x$, so that $s_i+x'-x\in\Omega$ for any $i=1,\dots, m$.
    Given $\xi'\in\partial\overline u^K(x')\cap K$, for any $y\in\Omega$, it holds
    \begin{equation*}
        u(y)\ge \overline u^K(y) \ge \overline u^K(x') + \xi'\cdot(y-x')
    \end{equation*}
    and thus, plugging $y=s_i+x'-x$, we deduce
    \begin{equation}\label{eq:local11}
        \overline u^K(x')\le u(s_i+x'-x) + \xi'\cdot(x-s_i) \fullstop
    \end{equation}
    Since $s_i\in S_\xi$, it holds $\nabla u(s_i) = \xi$ and therefore we have
    \begin{equation}\label{eq:local12}
        u(s_i+x'-x) = u(s_i) + \xi\cdot(x'-x) + \frac12(x'-x)\nabla^2 u(s_i)(x'-x) + \smallo(\abs{x'-x}^2)\fullstop
    \end{equation}
    Moreover, recalling that $\xi\in\partial\overline u^K(s_i)\cap\partial\overline u^K(x)$, we also have
    \begin{equation}\label{eq:local13}
        \overline u^K(s_i)-\overline u^K(x) = \xi\cdot(s_i-x) \fullstop
    \end{equation}
    Joining \cref{eq:local11,eq:local12,eq:local13} we obtain
    \begin{equation*}
        \overline u^K(x')\le \overline u^K(x) + \xi\cdot(x'-x) + \frac12(x'-x)\nabla^2 u(s_i)(x'-x) + \smallo(\abs{x'-x}^2) + (\xi'-\xi)\cdot(x-s_i) \fullstop
    \end{equation*}
    Multiplying this latter inequality with $\lambda_i$ and summing over $i=1,\dots,m$ we get
    \begin{equation*}
        \overline u^K(x')\le \overline u^K(x) + \xi\cdot(x'-x) + \frac12(x'-x)\left(\sum_{i=1}^m\lambda_i\nabla^2 u(s_i)\right)(x'-x) + \smallo(\abs{x'-x}^2) + (\xi'-\xi)\cdot v
    \end{equation*}
    and since $v\in N(\xi, K)$, we have shown \cref{eq:convex_envelope_fundamental} (notice that $\nabla^2 u(s_i)\ge 0$ as $s_i\in S_\xi$).
    
    As a consequence we get that $\nabla\overline u^K(x) = \xi$ and thus, recalling \cref{lem:preparation_convex_K}, we deduce also $\nabla\overline u^K(\R^n) = \nabla\overline u^K(\Omega) = K$.
    To conclude that $\overline u^K\in C^{1,1}$, let us observe that, for any $x\in\R^n$, the matrix $H_x$ of the statement can be found in
    \begin{equation*}
        \conv\Bigg(\nabla^2 u\bigg(\bigcup_{\xi\in K}S_\xi\bigg)\Bigg)
    \end{equation*}
    and the assumption $S_\xi\subseteq \Omega$ together with the compactness of $K$ implies that $\bigcup_{\xi\in K}S_\xi \compactsubset \Omega$. Therefore there is a constant $C=C(\Omega, u)$ such that for any $x\in\R^n$ it holds $H_x\le C\,\id$. Now it is standard to obtain $\overline u^K\in C^{1,1}(\R^n)$ (see, for example, \cite[Proposition 5.29]{ambrosio-carlotto-massacesi2018}).
\end{proof}
\begin{remark}\label{rem:envelope_locality}
    The proof of \cref{prop:convex_envelope_regularity} yields also a more \emph{local} result. Namely, if $S_\xi\subseteq\Omega$ for all $\xi\in U\cap K$ (where $U$ is an open set), then for any $x\in\R^n$ such that $\nabla \overline u^K(x)\in U$ there is a neighborhood of $x$ where $\overline u^K$ is $C^{1,1}$ and \cref{eq:convex_envelope_fundamental} holds (for some $H_x\in H(x, \nabla\overline u^K(x), K)$).
\end{remark}

\section{Construction of the coupling}\label{sec:coupling}
To better illustrate the ideas, before describing the construction of the coupling in the weighted setting, we construct the analogous coupling in the anisotropic case (without a weight). 
Even if the methods are the same, the construction in the anisotropic setting is less technical.
In order to do so, we have to define the anisotropic norm and the anisotropic perimeter (to get some further context about the anisotropic perimeter, see \cite[Introduction]{FMP}).

Given a compact convex set $K\subseteq\R^n$ such that $0_{\R^n}\in \mathring{K}$, let us define the anisotropic norm $\abs{\emptyparam}_{K^*}:\R^n\to\co0\infty$ as
\begin{equation*}
    \abs{v}_{K^*} = \sup\{v\cdot x:\ x\in K\} \fullstop
\end{equation*}
Then, the anisotropic perimeter of a set of finite perimeter $E\subseteq\R^n$ is defined as
\begin{equation*}
    \Per_K(E) = \int_{\bou E} \abs{\nu_{\bou E}}_{K^*} \de\Haus^{n-1} \fullstop
\end{equation*}
We have all the necessary definitions to construct the coupling for the anisotropic perimeter.
Before going on, let us remark that, in all the statements of this section, we need an additional assumption in dimension $2$ (namely, the indecomposability of the set $E$). This is because we need to approximate $E$ with \emph{connected} open sets and this is not always possible in dimension $2$.

\begin{theorem}[Anisotropic coupling]\label{thm:anisotropic_coupling}
    Let $E\subseteq\R^n$ be a set of finite perimeter (with $0<\abs{E}<\infty$) and let $K\subseteq\R^n$ be a compact convex set whose interior contains the origin. If $n=2$, we assume also that $E$ is indecomposable (see \cref{def:indecomposable}).
    
Then, there exists a $C^{1,1}$ convex function $\varphi:\R^n\to\R$ that satisfies $\nabla\varphi(\R^n)=\nabla\varphi(E) = K$ (up to negligible sets) and $\lapl\varphi\le \frac{\Per_K(E)}{\abs{E}}$.
\end{theorem}

\begin{proof}
    First, we prove the statement when $E$ is a bounded connected open set with smooth boundary and then we remove the assumption by approximation.
    
    Let $\nu:\partial E\to\R^n$ be the outer normal to the boundary. 
As in \cite{CRS}, let us consider the Neumann problem for $u:E\to\R$
    \begin{equation*}
        \begin{cases}
        \lapl u = \frac{\Per_K(E)}{\abs{E}} &\text{in $E$} \\
        \partial_\nu u = \abs{\nu}_{K^*} &\text{on $\partial E$}\fullstop
        \end{cases}
    \end{equation*}
    Since the compatibility condition $\int_E \lapl u = \int_{\partial E}\partial_\nu u$ is satisfied and $E$ is connected, the problem admits a solution $u$, which is smooth up to the boundary.
    
    For any $\xi\in\mathring{K}$ the minimum of the function $u(x)-\xi\cdot x$ cannot be attained on the boundary of $E$, indeed $\nabla(u(x)-\xi\cdot x)\cdot\nu = \abs{\nu}_{K^*} - \xi\cdot\nu > 0$. 
    
    Take a sequence of compact convex sets $(K_i)_{i\in\N}$ such that $K_i\compactsubset \mathring K_{i+1}$ and $\cup_{i\in\N}K_i = \mathring K$.
    It is not hard to see that $\overline u^{K_i}\nearrow\overline u^K$ locally uniformly. 
    Moreover, at any point $x\in E$ such that $\nabla^2 u(x)\ge 0$, it holds
    \begin{equation*}
        0 \le \nabla^2 u(x) \le \lapl u(x)\,\id = \frac{\Per_K(E)}{\abs{E}}\,\id
    \end{equation*}
    and thus, thanks to \cref{prop:convex_envelope_regularity}, the family $(\overline u^{K_i})_{i\in\N}$ is uniformly bounded in $C^{1,1}$.
Hence, $\overline u^K\in C^{1,1}$, and the convergence to $\overline u^K$ is (up to subsequence) in $C^1_{loc}(\R^n)$. 
Therefore, $\nabla \overline u^K(\R^n)=\nabla\overline u^K(\overline E)=K$.

Now, for any $i\in\N$ and almost any $x\in\R^n$, the Hessian $\nabla^2 \overline u^{K_i}$ exists and is controlled by $H_x\in H(x, \nabla \overline u^{K_i}, K_i)$. It follows that $\lapl\overline u^{K_i} \le \frac{\Per_K(E)}{\abs{E}}$ (since such an estimate holds for $u$). Sending $i\to\infty$ we deduce that the same holds for $\overline u^K$, thus $\varphi\defeq\overline u^K$ satisfies all the requirements.
    
    It remains to drop the regularity assumption on $E$.
    Let $(E_i)_{i\in\N}\subseteq\R^n$ be a sequence of \emph{connected} bounded open sets with smooth boundary such that $\Per_K(E_i)\to\Per_K(E)$ and $\abs{E_i\triangle E}\to 0$. In dimension $2$, the existence of this sequence is guaranteed by \cref{prop:indecomposable}; in higher dimension we can apply \cite[Theorem 13.8]{maggi2012} and then notice that the additional requirement of connectedness can be fulfilled adding a finite number of thin pipes between the connected components. 
    For each $E_i$, let $\varphi_i:\R^n\to\R$ be a function that satisfies the assumptions of the statement.
    
    Without loss of generality, thanks to Arzelà–Ascoli theorem, we can assume that $\varphi_i\to\varphi$ locally in $C^1$, where $\varphi:\R^n\to\R$ is a $C^{1,1}$ convex function such that $\nabla\varphi(\R^n)\subseteq K$ and $\lapl\varphi \le \frac{\Per_K(E)}{\abs{E}}$.
    To conclude it is sufficient to prove $\abs{\nabla\varphi(E)} \ge \abs{K}$.
    
    Let $C\subseteq E$ be a compact set.
    It holds
    \begin{equation}\label{eq:estimate_outside}
        \abs{\nabla\varphi_i(E_i\setminus C)}\le \Lip(\nabla\varphi_i)^n \abs{E_i\setminus C}
        \le \gamma \cdot (\abs{E_i\triangle E} + \abs{E\setminus C}) \comma
    \end{equation}
    where $\gamma = \gamma(n, K, E) = \left(2n\frac{\Per_K(E)}{\abs{E}}\right)^n$. Since $\nabla \varphi_i(E_i) = K$, it holds $K\setminus \nabla\varphi_i(E_i\setminus C)\subseteq \nabla\varphi_i(C)$ and therefore \cref{eq:estimate_outside} implies
    \begin{equation}\label{eq:estimating_for_i}
        \abs{\nabla\varphi_i(C)} \ge \abs{K}-\gamma \cdot (\abs{E_i\triangle E} + \abs{E\setminus C}) \fullstop
    \end{equation}
    Take any $y\in \limsup \nabla\varphi_i(C)$ and let $x_i\in C$ be such that $\nabla\varphi_i(x_i)=y$ (we avoid passing to a subsequence for notational simplicity). By compactness we know that $x_i\to x\in C$ and since the convergence $\varphi_i\to\varphi$ is locally in $C^1$, it follows that $\nabla\varphi(x)=y$. Hence, applying \cref{eq:estimating_for_i}, we deduce
    \begin{equation*}
        \abs{\nabla \varphi(E)}\ge \abs{\nabla \varphi(C)} \ge \abs*{\limsup \nabla\varphi_i(C)}
        \ge \limsup \abs{\nabla\varphi_i(C)} \ge \abs{K}-\gamma\cdot\abs{E\setminus C} \fullstop
    \end{equation*}
    The conclusion now follows as $\abs{E\setminus C}$ can be chosen arbitrarily small.
\end{proof}

\begin{remark}
    Notice that the existence of the coupling implies the anisotropic isoperimetric inequality, exactly as in \cref{eq:proof_iso_abp}.
\end{remark}

It is now time to construct the convex coupling in the weighted setting. As anticipated, the main idea (taking the convex envelope of the solution of an elliptic problem) is unchanged with respect to \cref{thm:anisotropic_coupling}. 
On the other hand, the proof that the coupling works for smooth sets/weights and the approximation arguments needed to handle any set/weight are more demanding. 
Notice that at this level we require $w\equiv 0$ on $\partial\Sigma$; later on we remove this assumption.

\begin{theorem}[Weighted coupling]\label{thm:weighted_coupling}
    Consider $n,\,\alpha,\,\Sigma,\,w$ satisfying \cref{eq:assumptions} and assume moreover $w\equiv 0$ on~$\partial\Sigma$.
    Let $E\subseteq\Sigma$ be a set of finite $w$-perimeter with $0<w(E)<\infty$. If $n=2$, we assume also that $E$ is $w$-indecomposable (see \cref{def:w_indecomposable}).
    Then, there exists a $C^{1,1}$ convex function $\varphi:\R^n\to\R$ that satisfies $\nabla\varphi(\R^n)=\nabla\varphi(E) = B_1\cap\Sigma$ (up to negligible sets) and 
    \begin{equation}\label{eq:weighted_fundamental_control}
        \lapl\varphi + \alpha\left(\frac{w(\nabla\varphi)}{w}\right)^{\frac1\alpha}
        \le \frac{\Per_w(E)}{w(E)}\fullstop
    \end{equation}
\end{theorem}

\begin{proof}
    We start by showing that the result holds if $E$ is a bounded open set with smooth boundary such that $\overline E\compactsubset\Sigma$ and the weight $w$ is smooth in $E$.
    Let $b_E=\Per_w(E)/w(E)$  and let $K\defeq \overline{B_1\cap\Sigma}$.
    
    Let $\nu:\partial E\to\R^n$ be the outer normal to the boundary. 
As in \cite{CRS}, let us consider the following Neumann problem for $u:E\to\R$
    \begin{equation}\label{eq:weighted_elliptic_problem}
        \begin{cases}
        \div(w\nabla u) = w\, b_E &\text{in $E$} \\
        \partial_\nu u = 1 &\text{on $\partial E$}\fullstop
        \end{cases}
    \end{equation}
    Since the compatibility condition $\int_E \div(w\nabla u) = \int_{\partial E}w \partial_\nu u$ is satisfied, the problem admits a solution.
    
    With the exact same reasoning adopted in the proof of \cref{thm:anisotropic_coupling}, we deduce that for any $\xi\in\mathring K$, the minimum of $u(x)-\xi\cdot x$ cannot be attained on the boundary of $E$. Thence $\overline u^K\in C^{1,1}$ and $\nabla \overline u^K(\R^n) = \nabla\overline u^K(\overline E) = K$.
    
    Let $K_r\defeq \overline{\Sigma\cap B_r}$. 
Since $K_r\nearrow K$ as $r\nearrow 1$, then it holds $\overline u^{K_r} \nearrow \overline u^K$ locally uniformly as $r\nearrow 1$. 
Notice that $K_r\not\subseteq\mathring{K}$, so we cannot apply directly \cref{prop:convex_envelope_regularity}. 
Our goal is to prove \cref{eq:weighted_fundamental_control} with $\varphi=\overline u^{K_r}$ for any $0<r<1$. We consider two cases, depending on whether $\nabla\overline u^{K_r}$ belongs to the boundary or to the interior of $\Sigma$.
    
    If $\nabla\overline u^{K_r}(x)\in\partial\Sigma$, since $w(\partial\Sigma)=0$, it is sufficient to show that $\lapl\overline u^{K_r}(x)\le b_E$. Applying \cref{eq:weighted_elliptic_problem,lem:5.1}, we have $\lapl u(x)\le b_E$ at all points $x\in E$ such that $\nabla u(x)\in\Sigma$. Then the same inequality for $\overline u^{K_r}$ follows repeating the approximation argument contained at the beginning of the proof of \cref{thm:anisotropic_coupling} (recall that we cannot apply directly \cref{prop:convex_envelope_regularity}).
    
    On the other hand, if $\xi\defeq \nabla\overline u^{K_r}(x)\not\in\partial\Sigma$, then it belongs to $\mathring K$ and we can apply \cref{prop:convex_envelope_regularity} (see \cref{rem:envelope_locality}). Thus $\nabla^2 \overline u^{K_r}(x) \le H_x \in H(x, \xi, K_r)$. Let $H_x=\sum_{i=1}^m \lambda_i\nabla^2 u(s_i)$ with $s_i\in S_\xi$ and $x-\sum \lambda_i s_i = v\in N(\xi, K_r)$. Notice that, since $\Sigma$ is convex, $x-v\in\Sigma$. Applying \cref{lem:5.1} we obtain
    \begin{align}\label{eq:local55}\begin{split}
        \alpha\left(\frac{w(\xi)}{w(x-v)}\right)^{\frac1\alpha}
        &\le 
        \alpha w(\xi)^{\frac1\alpha} \sum_{i=1}^m\lambda_i w(s_i)^{-\frac1\alpha}
        = 
        \sum_{i=1}^m\lambda_i \alpha \left(\frac{w(\xi)}{w(s_i)}\right)^{\frac1\alpha} \\
        &\le 
        \sum_{i=1}^m\lambda_i \frac{\nabla w(s_i)\cdot \xi}{w(s_i)}
        =
        \sum_{i=1}^m\lambda_i \frac{\nabla w(s_i)\cdot \nabla u(s_i)}{w(s_i)} \fullstop
    \end{split}\end{align}
    Since $u$ is a solution of \cref{eq:weighted_elliptic_problem}, it holds
    \begin{equation}\label{eq:local56}
        \lapl u(s_i) + \frac{\nabla w(s_i)\cdot \nabla u(s_i)}{w(s_i)} = b_E \fullstop
    \end{equation}
    Combining \cref{eq:local55,eq:local56}, we deduce
    \begin{equation*}
        \sum_{i=1}^m \lambda_i\lapl u(s_i) + \alpha\left(\frac{w(\xi)}{w(x-v)}\right)^{\frac1\alpha} \le b_E\comma
    \end{equation*}
    and thus
    \begin{equation*}
        \nabla^2\overline u^{K_r}(x) + \alpha\left(\frac{w(\nabla\overline u^{K_r})}{w(x-v)}\right)^{\frac1\alpha} \le b_E \fullstop
    \end{equation*}
    This latter inequality is \emph{almost} the desired one, but for the presence of $w(x-v)$ instead of $w(x)$. To fix this issue we notice that if $\xi\in\mathring K_r$ then $v=0$. Otherwise $\xi\in\partial K_r\cap\mathring K$ and therefore $\xi\in \Sigma\cap\partial B_r$ and $N(\xi, K_r)=\{\xi\}$. Hence $v=\xi\in\Sigma$ and, recalling once again \cref{lem:5.1}, it holds $w(x-v)\le w(x)$.
    
    We managed to prove \cref{eq:weighted_fundamental_control} for $\varphi=\overline u^{K_r}$ and since $\overline u^{K_r}\to \overline u^K$ locally in $C^1$ as $r\nearrow 1$, it follows that the same is true also for $\overline u^K$.
    
    It remains to drop the regularity assumption on $E$ and $w$. 
    First, we drop the regularity assumption on $w$. 
Thanks to \cref{lem:smoothing_weight} we can find a sequence of admissible weights $w_i\to w$ that are smooth in $E$ and that converge to $w$ locally uniformly. For any such $w_i$ we can find $\varphi_i:\R^n\to\R$ such that the statement holds. In particular the family $\varphi_i$ is bounded in $C^{1,1}$ and therefore, up to subsequence, converges locally in $C^1$ to $\varphi:\R^n\to\R$. It is not hard to check that $\varphi$ satisfies all the desired properties.
    
    Finally, the method used in the last part of the proof of \cref{thm:anisotropic_coupling} can be applied also here to remove the assumption that $E$ is compactly contained in $
    \Sigma$ and has a smooth boundary (applying \cref{prop:w_indecomposable} instead of \cref{prop:indecomposable}). 
    Notice that in this final step we need $\Per_w(E_i)\to \Per_w(E)$, with $E_i\compactsubset \Sigma$, which holds since $w\equiv 0$ on $\partial\Sigma$.
\end{proof}


The next proposition is fundamental for the proof of the stability of the isoperimetric inequality; it contains all the properties of the coupling that we will need.
In this proposition we drop the additional assumption $w\equiv 0$ on $\partial\Sigma$ that was necessary for \cref{thm:weighted_coupling}.

\begin{proposition}\label{prop:weighted_coupling_properties}
    Consider $n,\,\alpha,\,\Sigma,\,w$ satisfying \cref{eq:assumptions} and let $E\subseteq\Sigma$ be a set of finite $w$-perimeter with $w(E)=1$ and $\delta_w(E)\le 1$.
    If $n=2$, we assume also that $E$ is $w$-indecomposable (see \cref{def:w_indecomposable}).
    There is a $C^{1,1}$ convex function $\varphi:\R^n\to\R$ such that $\norm{\nabla^2\varphi}_{\infty}\lesssim 1$, $\nabla\varphi(\R^n)=\nabla\varphi(E)=\Sigma\cap B_1$ (up to negligible sets), and
    \begin{align}
        &\int_E \abs{\nabla^2\varphi -\id}\,w \lesssim \delta_w(E)^{\frac12} \label{eq:hessian_control} \comma\\
        &
        \int_{\bou E\cap\Sigma} \big(1-\abs{\nabla\varphi}\big)\,w\de\Haus^{n-1}
        \lesssim
        \delta_w(E) \fullstop\label{eq:boundary_control}
    \end{align}
    Moreover, for any $Q\compactsubset\Sigma$, we have
    \begin{equation} \label{eq:weight_control}
        \int_{E\cap Q} \abs*{w(\nabla\varphi)^{\frac1\alpha}-w^{\frac1\alpha}} \le C\delta_w(E)^{\frac12} \comma
    \end{equation}
where $C$ is a constant depending only on $n$, $\alpha$, $\Sigma$, and $Q$.
\end{proposition}
\begin{proof}
    First, assuming that $w\equiv 0$ on $\partial\Sigma$, we prove that the map built in \cref{thm:weighted_coupling} satisfies all the constraints. Then we remove the additional assumption by approximation.
    
    Let us assume $w(\partial\Sigma) = 0$ and let $\varphi:\R^n\to\R$ be the map built in \cref{thm:weighted_coupling}.
    
    Applying the area formula, the arithmetic-geometric mean inequality, and the properties of $\varphi$ described in \cref{thm:weighted_coupling}, we get
    \begin{align}\begin{split}\label{eq:local_weight_prop0}
        1&=w(B_1\cap\Sigma)=w(\nabla\varphi(E))
        \le \int_E \det(\nabla^2\varphi)w(\nabla\varphi)
        =\int_E \det(\nabla^2\varphi)\Big(\frac{w(\nabla\varphi)}{w}\Big)w \\
        &\le 
        \int_E \Big(\frac{\lapl \varphi}n\Big)^n\Big[\Big(\frac{w(\nabla\varphi)}{w}\Big)^{\frac1\alpha}\Big]^{\alpha}w
        \le
        \int_E \bigg(\frac{\lapl\varphi + \alpha \big(\frac{w(\nabla\varphi)}{w}\big)^{\frac1\alpha}}{D}\bigg)^D w 
        \le 
        \int_E \Big(\frac{\Per_w(E)}{D}\Big)^D w \\
        &= \big(1+\delta_w(E)\big)^D \comma
    \end{split}
    \end{align}
    where we used $s^nt^\alpha\le \big(\frac{n s + \alpha t}{n+\alpha}\big)^{n+\alpha}$.
    Thus
    \begin{equation}\label{eq:local_weight_prop1}
        \int_E \left(\Big(\frac{\Per_w(E)}{D}\Big)^D - \det(\nabla^2\varphi)\Big(\frac{w(\nabla\varphi)}{w}\Big) \right)w 
        \le \big(1+\delta_w(E)\big)^D-1 \lesssim \delta_w(E) \fullstop
    \end{equation}
    Applying \cref{lem:quantitative_amgm} with $\lambda=(\rlap{$\overbrace{\phantom{1,\dots, 1}}^n$}1,\dots,1,\alpha)$, $(x_1,\dots,x_n)$ equal to the eigenvalues of $\nabla^2\varphi$, \ $x_{n+1}=(\frac{w(\nabla\varphi)}{w})^{\frac1\alpha}$, and $c=\frac{\Per_w(E)}D$, we obtain
    \begin{equation}\label{eq:local_weight_prop2}
        \abs{\nabla^2 \varphi - \id}^2 + 
        \abs*{\Big(\frac{w(\nabla\varphi)}{w}\Big)^{\frac1\alpha}-1}^2
        \lesssim \Big(\frac{\Per_w(E)}{D}\Big)^D - \det(\nabla^2\varphi)\Big(\frac{w(\nabla\varphi)}{w}\Big)\fullstop
    \end{equation}
    Joining \cref{eq:local_weight_prop1,eq:local_weight_prop2} we obtain
    \begin{equation*}
        \int_E \abs{\nabla^2 \varphi - \id}^2 w 
        +
        \int_E \abs*{\Big(\frac{w(\nabla\varphi)}{w}\Big)^{\frac1\alpha}-1}^2 w
        \lesssim \delta_w(E) \fullstop
    \end{equation*}
    This implies (using Cauchy-Schwarz) \cref{eq:hessian_control} and 
    \begin{equation*}
        \int_E \abs*{\Big(\frac{w(\nabla\varphi)}{w}\Big)^{\frac1\alpha}-1}\,w \lesssim \delta_w(E)^{\frac12} \fullstop
    \end{equation*}
    The weight $w$ is bounded in $Q$ from below and above by constants that depend only on $n,\alpha,\Sigma,Q$ because $w^{\frac1\alpha}$ is concave and $w(B_1\cap\Sigma)=1$. Thus, the last estimate implies \cref{eq:weight_control}.
    
    Proceeding as in \cref{eq:local_weight_prop0}, applying \cref{lem:5.1} and the divergence theorem, we obtain
    \begin{align*}
        1 
        &\le 
        \int_E \bigg(\frac{\lapl\varphi + \alpha \big(\frac{w(\nabla\varphi)}{w}\big)^{\frac1\alpha}}{D}\bigg)^D w
        \le \frac{\Per_w(E)^{D-1}}{D^D}\int_E \bigg(\lapl\varphi + \alpha \Big(\frac{w(\nabla\varphi)}{w}\Big)^{\frac1\alpha}\bigg)w \\
        &\le \frac{\Per_w(E)^{D-1}}{D^D}\int_E \div(w\nabla\varphi)
        =
        \frac{\Per_w(E)^{D-1}}{D^D}\int_{\bou E} (\nabla\varphi\cdot\nu_{\bou E})w\de\Haus^{n-1}  \\
        &\le 
        \frac{\Per_w(E)^D}{D^D}-\frac{\Per_w(E)^{D-1}}{D^D}\int_{\bou E} (1-\abs{\nabla\varphi})w\de\Haus^{n-1}\comma
    \end{align*}
    where in the last step we applied $\nabla\varphi\cdot\nu \le \abs{\nabla\varphi}\le 1$ (recall that $\nabla\varphi\in B_1\cap\Sigma$).
    The estimate \cref{eq:boundary_control} follows rearranging the terms.
    
    It remains to drop the assumption $w(\partial\Sigma)=0$.
    Thanks to \cref{lem:weight_approximation_zero_trace} we can find a sequence $(w_i)_{i\in\N}$ of $\alpha$-homogeneous weights such that $w_i\equiv 0$ on $\partial\Sigma$, and $w_i\to w$ uniformly on compact subsets of $\Sigma$. Let $\varphi_i:\R^n\to\R$ be the map built in \cref{thm:weighted_coupling} when the weight is $w_i$.

    Notice that $\delta_{w_i}\to\delta_w$ as $i\to\infty$ and therefore the family $(\varphi_i)_{i\in\N}$ is bounded in $C^{1,1}$. Up to subsequence, we can assume that $\varphi_i\to\varphi$ locally in $C^1$, where $\varphi:\R^n\to\R$ is a $C^{1,1}$-function. 
    Following the method adopted in the proof of \cref{thm:weighted_coupling}, we can establish $\nabla\varphi(\R^n)=\nabla\varphi(E)=\Sigma\cap B_1$. The properties \cref{eq:hessian_control,eq:boundary_control} follow directly from the convergence of weights and maps.
    On the other hand, passing to the limit in~\cref{eq:weight_control}, we need to avoid $\nabla\varphi\in\partial\Sigma$ because in that case $w_i(\nabla\varphi_i)$ may not converge to $w(\nabla\varphi)$. Hence, we obtain
    \begin{equation}\label{eq:tmp6721}
        \int_{E\cap Q\cap \{\nabla\varphi\not\in\partial\Sigma\}}
        \abs*{w(\nabla\varphi)^{\frac1\alpha}-w^{\frac1\alpha}} \le C \delta_w(E)^{\frac12} \fullstop
    \end{equation}
    The area formula for Lipschitz functions implies $\nabla^2\varphi$ is singular almost everywhere in $\{\nabla\varphi\in\partial\Sigma\}$, in particular $\abs{\nabla^2\varphi-\id}\gtrsim 1$. Therefore \cref{eq:hessian_control} implies 
    \begin{equation*}
        w\big(Q\cap\{\nabla\varphi\in\partial\Sigma\}\big) \lesssim \delta_w(E)^{\frac12}\comma
    \end{equation*}
    and this is sufficient to deduce \cref{eq:weight_control} from \cref{eq:tmp6721}.
\end{proof}

\section{Ruling out translational freedom}\label{sec:ruling_out_translations}

Recall that we assume $\Sigma=\R^k\times\widetilde\Sigma$ where $\widetilde\Sigma\subseteq\R^{n-k}$ is an open convex cone containing no lines.
Moreover, we can split $\R^{n-k}$ as $\R^h\times\R^{n-k-h}$, where $\R^h$ identifies the directions along which $w$ is constant, namely, the vectors $\xi$ such that $w(x+t\xi)=w(x)$ whenever both sides make sense ($x\in\Sigma$, $t\in\R$).

Let 
\begin{equation}\label{eq:LCE}
\begin{split}\mathcal L &\defeq \R^k\times \{0_{\R^{n-k}}\}\,\textrm{ be the subspace of $\mathcal{L}$ines contained in $\Sigma$,}\\
\mathcal C & \defeq\{0_{\R^k}\}\times\R^h\times\{0_{\R^{n-k-h}}\}\,\textrm{ be the subspace of direction of $\mathcal C$onstancy for $w$,}\\
\mathcal E &\defeq \{0_{\R^{h+k}}\}\times\R^{n-k-h}\,\textrm{ be everything $\mathcal E$lse.}
\end{split}
\end{equation}
Notice that if $\mathcal C\neq \{0_{\R^n}\}$ then $w\not\equiv 0$ on $\partial\Sigma$.

\Cref{prop:magic_compact_set} corresponds to the strategy \cref{it:intro_nonconstant} described in the introduction to prove $\abs{\widetilde x_0} \le C \delta_w(E)^{\frac12}$.
\Cref{lem:linear_growth_ball} implements the strategy \cref{it:intro_ball_growth}. Finally, \cref{it:intro_ball_growth} and \cref{it:intro_constant_hard} are condensed in \cref{prop:controlling_constant_directions}.

We will need the following two technical lemmas.
\begin{lemma}\label{lem:elementary_boring_result}
    Let $\eta:\R\to\R$ be any function. Then for any $a<b$ and $\eps>0$ it holds
    \begin{equation*}
        \int_a^b\abs{\eta(t+\eps)-\eta(t)} \de t\ge
        \eps \Big(\inf_{\abs{t-b}\le\eps}\eta(t) - \sup_{\abs{t-a}\le \eps}\eta(t) \Big) \fullstop
    \end{equation*}
\end{lemma}
\begin{proof}
    Let $m=\lfloor\frac{\abs{a-b}}{\eps}\rfloor$, where $\lfloor x\rfloor$ denotes the integer part of $x$. 
    The triangle inequality implies
    \begin{align*}
        \int_a^b \abs{\eta(t+\eps)-\eta(t)} \de t
        \ge\int_0^\eps \sum_{i=1}^m \abs{\eta(a+t+i\eps)-\eta(a+t+(i-1)\eps)} \de t
        \ge \int_0^\eps \abs{\eta(a+t+m\eps)-\eta(a+t)}\de t\comma
    \end{align*}
    and the desired inequality follows since $\abs{a+t+m\eps-b}\le \eps$ for any $0\le t\le \eps$.
\end{proof}

The second one reads as follows.

\begin{lemma}\label{lem:derivative_volume}
    Given $f\in BV_{loc}(\R^n)$ and $\eta\in L^{\infty}(\R^n)$ with compact support, the function $\psi:\R^n\to\R$ defined as
    \begin{equation*}
        \psi(\xi) \defeq \int_{\R^n} \eta(x) f(x+\xi)\de x
    \end{equation*}
    is locally Lipschitz continuous, hence differentiable almost everywhere, and its gradient satisfies (almost everywhere)
    \begin{equation*}
        \nabla \psi(\xi) = \int_{\R^n} \eta(x-\xi) \de \diff f(x) \fullstop
    \end{equation*}
\end{lemma}
\begin{proof}
    Let $\varphi\in C^{\infty}_c(\R^n)$ be a smooth function with compact support.
    From multiple applications of Fubini's theorem and the divergence theorem for BV functions \cref{eq:intro_div_theorem}, it follows
    \begin{equation*}
        \int_{\R^n}\nabla\varphi(\xi)\psi(\xi)\de \xi = 
        -\int_{\R^n}\varphi(\xi)\Big(\int_{\R^n}\eta(x-\xi)\de\widetilde\nabla f(x)\Big)\de\xi \fullstop
    \end{equation*}
    Since $\int_{\R^n}\eta(x-\xi)\de\widetilde\nabla f(x)$ is locally bounded, $\psi$ is locally Lipschitz continuous and $\int_{\R^n}\eta(x-\xi)\de\widetilde\nabla f(x)$ is its weak gradient (hence its classical gradient almost everywhere).
\end{proof}

Next, we prove the following.

\begin{proposition}\label{prop:magic_compact_set}
    Consider $n,\,\alpha,\,\Sigma,\,w$ satisfying \cref{eq:assumptions}, and let $\mathcal L$, $\mathcal C$, and $\mathcal E$ be as in \cref{eq:LCE}.
    There exists an $\hat\eps=\hat\eps(n,\alpha,\Sigma,w)>0$ and a compact set $\hat Q\subseteq B_{\frac12}\cap\Sigma$ (that depends on $n,\alpha,\Sigma, w$) such that
    \begin{enumerate}[label=(\arabic*)]
        \item \label{it:magic_compact_set_far} $d(\hat Q, \partial\Sigma) > 2\hat\eps$, 
        \item \label{it:magic_compact_set} For any $\xi\in B_{\hat\eps}$ it holds
        \begin{equation*}
            \int_{\hat Q} \abs{w^{\frac1\alpha}(x+\xi)-w^{\frac1\alpha}(x)}\de x \ge \hat\eps\abs{\pi_{\mathcal E}(\xi)} \comma
        \end{equation*}
        where $\pi_{\mathcal E}:\R^n\to\mathcal E$ is the orthogonal projection onto $\mathcal E$.
    \end{enumerate}
\end{proposition}

\begin{proof}
    Given a compact set $Q\subseteq \Sigma$, an $\eps>0$ and a subset $U\subseteq \partial B_1\cap \mathcal E$ we say that the pair $(U,\eps)$ is compatible with the compact set $Q$ if $d(Q,\partial\Sigma) > 2\eps$ and
    \begin{equation*}
        \int_{Q} \abs{w^{\frac1\alpha}(x+\xi)-w^{\frac1\alpha}(x)}\de x \ge \eps\abs{\xi}
    \end{equation*}
    for any $\xi\in B_\eps$ such that $\xi/\abs{\xi}\in U$
    
    Given two points $p_1,p_2\in B_{\frac12}\cap\Sigma$ and $r>0$, let us define the compact set 
    \begin{equation*}
        Q(p_1, p_2, r) \defeq \big\{x\in \Sigma\cap \overline B_{\frac12}:\ d(x, \partial\Sigma)\ge r,\ d(x, [p_1,p_2])\le r\big\}\comma
    \end{equation*}
    where $[p_1,p_2]$ denotes the segment between $p_1$ and $p_2$. 
    
    Given a direction $\theta \in \partial B_1\cap \mathcal E$, take two points $p_\theta, p'_\theta \in B_{\frac12}\cap\Sigma$ such that $w(p_\theta)\not=w(p'_\theta)$ and $p_\theta-p'_\theta$ is a multiple of $\theta$ (the existence of the two points follows from $\theta\in\mathcal E$).
    Applying \cref{lem:elementary_boring_result}, we can find a small $r_\theta>0$, a small $\eps_\theta>0$ and a neighborhood $U_\theta\subseteq \partial B_1\cap \mathcal E$ of $\theta$ such that $(\eps_\theta, U_\theta)$ is compatible with $Q(p_\theta, p'_\theta, r_\theta)$. Since $\partial B_1\cap\mathcal E$ is compact, we can find $U_{\theta_1},\dots,U_{\theta_k}$, with $k\in\N$, that cover $\partial B_1\cap\mathcal E$. Hence, it is not hard to check that $\hat\eps\defeq \min(\eps_{\theta_1},\dots,\eps_{\theta_k})$ and
    \begin{equation*}
        \hat Q \defeq \bigcup_{i=1}^k Q(p_{\theta_i}, p'_{\theta_i}, r_{\theta_i})
    \end{equation*}
    satisfy \cref{it:magic_compact_set_far,it:magic_compact_set} with the additional constraint $\xi\in\mathcal E$.
    The whole statement follows because the left-hand side of \cref{it:magic_compact_set} does not change if we replace $\xi$ with $\pi_{\mathcal E}(\xi)$.
\end{proof}

We will also need the following.

\begin{lemma}\label{lem:linear_growth_ball}
    Consider $n,\,\alpha,\,\Sigma,\,w$ satisfying \cref{eq:assumptions}, and let $\mathcal L$, $\mathcal C$, and $\mathcal E$ be as in \cref{eq:LCE}.
    Then, there is a small constant $c_1=c_1(n,\alpha,\Sigma, w)>0$ such that
    \begin{equation*}
        w(B_1(\xi)\cap\Sigma) \ge w(B_1\cap\Sigma) + c_1\abs{\xi}\comma
    \end{equation*}
for all $\xi\in B_{c_1}\cap\mathcal C$ satisfying $d(\frac{\xi}{\abs{\xi}}, \Sigma) <c_1$.
\end{lemma}
\begin{proof}
    For any $\xi\in\mathcal C$ and any $x\in\Sigma$ with $x+\xi\not\in\Sigma$, let us define $w(x+\xi)\defeq w(x)$ (we are extending the domain of $w$ exploiting its constancy in certain directions). The definition is consistent because of the condition $\xi\in\mathcal C$.
    
    Notice that, for any $\xi\in\mathcal C$, it holds
    \begin{equation*}
        w(B_1(\xi)\cap\Sigma) = \int_{B_1}\characteristic\Sigma(x+\xi)w(x+\xi)\de x
        = \int_{B_1}w(x)\characteristic{\Sigma}(x+\xi)\de x \fullstop
    \end{equation*}
    Applying \cref{lem:derivative_volume} with $f=\characteristic\Sigma$ and $\eta=w\,\chi_{B_1}$, we deduce that for $\xi\in\mathcal C$ it holds
    \begin{equation}\label{eq:tmp5461}
        w(B_1(\xi)\cap\Sigma) - w(B_1\cap\Sigma) \ge - \overline\nu\cdot\xi -\smallo(\abs{\xi}) \comma
    \end{equation}
    where 
    \begin{equation*}
        \overline\nu\defeq \int_{\partial\Sigma\cap B_1} \nu_{\partial\Sigma}\,w\de\Haus^{n-1}
    \end{equation*}
    and $\nu_{\partial\Sigma}$ is the \emph{outer} normal to $\partial\Sigma$.
    If $\xi\in\mathcal C\cap\Sigma$, it holds $\nu_\Sigma \cdot \xi<0$ at any point and therefore $\overline \nu\cdot\xi < 0$ (recall that if $\mathcal C\not=\{0\}$ then $w\not\equiv 0$ on $\partial\Sigma$). 
    The same strict inequality holds also if $\xi\in\mathcal C\cap\partial\Sigma\setminus\{0\}$ (because by definition $\mathcal C\cap\Sigma$ does not contain lines). 
    Hence, there is a small $c>0$ such that for any $\xi\in\mathcal C\cap\overline\Sigma$ it holds
    \begin{equation}\label{eq:tmp67992}
        \overline\nu\cdot\xi \le -c\abs{\xi} \fullstop
    \end{equation}
    The two estimates \cref{eq:tmp5461,eq:tmp67992}, together with a simple continuity argument, yield the desired result.
\end{proof}

Finally, we prove the following.
\begin{proposition}\label{prop:controlling_constant_directions}
    Consider $n,\,\alpha,\,\Sigma,\,w$ satisfying \cref{eq:assumptions}, and let $\mathcal L$, $\mathcal C$, and $\mathcal E$ be as in \cref{eq:LCE}.
    There is a small constant $c=c(n,\alpha,\Sigma, w)>0$ such that for any $\xi\in B_c\cap\mathcal C$ at least one of the following two statements is true:
    
\vspace{1mm}

    \begin{enumerate}[ref=(\arabic*)]
        \item It holds $\abs{w(\Sigma\cap B_1(\xi))-w(\Sigma\cap B_1)}\ge c\abs{\xi}$.\label{it:linear_growth_ball}
        
\vspace{2mm}        
        
        \item There exists a set $S\subseteq \langle \xi\rangle ^{\perp}$ such that, setting $t_0(s)$ and $t_1(s)$ so that $\Sigma_s \defeq \Sigma\cap\{s+t\xi:t\in\R\} = \big\{s+t\frac{\xi}{\abs{\xi}}:\ t_0(s) < t < t_1(s)\big\}$, the following hold.
        \label{it:good_slicing}
        \begin{enumerate}[ref=(\alph*)]
            \item For any $s\in S$, it holds $-\infty < t_0(s) < t_0(s) + c < t_1(s) < +\infty$.\label{it:loca}
            \item For any $s\in S$ and any $z\in\Sigma_s$, it holds $w(z)>c$ (notice that $w$ is constant on $\Sigma_s$). \label{it:locb}
            \item It holds $\Haus^{n-1}(S) > c$. \label{it:locc}
            \item For any $s\in S$ and any $t<t_0(S)$ it holds $d(s + t\frac{\xi}{\abs{\xi}}, \Sigma) > c (t_0-t)$. \label{it:locd}
            \item For any $s\in S$ and any $t>t_1(S)$ it holds $d(s + t\frac{\xi}{\abs{\xi}}, \Sigma) > c (t-t_1)$. \label{it:loce}
            \item For any $s\in S$, $\Sigma_s\subseteq B_{\frac12}\cap\Sigma$. \label{it:locf}
        \end{enumerate}
    \end{enumerate}
\end{proposition}

\begin{proof}
    Let $c_1$ be the constant present in the statement of \cref{lem:linear_growth_ball}.
    If $d(\xi, \Sigma) < c_1\abs{\xi}$ or $d(\xi, -\Sigma) < c_1\abs{\xi}$, then \cref{lem:linear_growth_ball} implies that \cref{it:linear_growth_ball} holds (choosing $c$ smaller then $c_1$).
    
    We are going to show that if $d(\xi, \Sigma)>c_1\abs{\xi}$ and $d(\xi, -\Sigma)>c_1\abs{\xi}$, then we can find a set $S$ satisfying the requirements of \cref{it:good_slicing}.
    Given that $\xi$ is \emph{far} from $\Sigma$ and $-\Sigma$, the properties \cref{it:locd,it:loce} are satisfied independently of the choice of $S$.
    
    Fix an element $v\in\Sigma$ and a small real number $l>0$.
    Let $S$ be the image of the projection onto $\langle\xi\rangle^\perp$ of $B_{l^2}(lv)$.
    Up to choosing $c$ and $l$ sufficiently small, the properties \cref{it:locb,it:locc} are satisfied.
    
    The property \cref{it:locf} follows from the observation that for any $s\in S$ it holds $\Sigma_s\cap B_{2l}\not=\emptyset$. Similarly, the property \cref{it:loca} follows from the observation that for any $s\in S$ there is a point $z\in\Sigma_s$ such that $d(z, \partial\Sigma) > \eps_0l$, where $\eps_0>0$ is a small constant that depends only on the chosen vector $v$.
\end{proof}

\section{Spherical boundary implies ball}\label{sec:spherical_boundary_implies_ball}

The goal of this section is to prove the following.

\begin{proposition}\label{prop:close_to_translated_ball}
    Consider $n,\,\alpha,\,\Sigma,\,w$ satisfying \cref{eq:assumptions} and let $E\subseteq\Sigma$ be a set of finite $w$-perimeter with $w(E\cap B_{\frac12})\ge \frac12w(B_{\frac12}\cap \Sigma)$.
    Then, we have
    \begin{equation*}
        w\big(E\triangle (B_1(x_0)\cap\Sigma)\big) \lesssim \int_{\bou E\cap\Sigma} \abs{\abs{x-x_0}-1}w(x)\de\Haus^{n-1}(x)\comma
    \end{equation*}
    for any sufficiently small $x_0\in\R^n$.
\end{proposition}

The strategy is to obtain first a robust analogous result in $1$-dimension (that is \cref{lem:one_dim_boundary_to_inside}) and deduce the full statement through a polar slicing.

\begin{lemma}\label{lem:one_dim_boundary_to_inside}
    For any $\gamma\ge 0$, there is a constant $C_\gamma>0$ such that the following statement holds.
    
    Let $E\subseteq\co0\infty$ be a $1$-dimensional set of locally finite perimeter with $\abs{E} < \infty$. For any $1-\frac14\le l\le1+\frac14$ it holds
    \begin{equation}\label{eq:no_fantasy}
        \int_{E\triangle\cc0l} t^{\gamma}\de t 
        \le C_\gamma \left(
        \int_{\cc0{\frac12}\setminus E} t^{\gamma}\de t
        +\int_{\bou E} t^{\gamma}\abs{l-t}\de\Haus^0(t) \right)\fullstop
    \end{equation}
\end{lemma}

\begin{proof}
    For simplicity, we prove it only for $l=1$; the proof works (up to slightly increasing the value of the constant $C_\gamma$) also for any other $1-\frac14\le l\le 1+\frac14$.
    It holds $E\triangle \cc01=\big(\oo1\infty\cap E\big)\cup\big(\cc01\setminus E\big)$.
    
    We will estimate independently the two terms
    \begin{equation*}
        \int_{\co1\infty\cap E} t^\gamma\de t
        \quad\text{and}\quad
        \int_{\cc{\frac12}{1}\setminus E} t^{\gamma} \de t \fullstop
    \end{equation*}
    Notice that if $\co1\infty\cap E=\emptyset$, then the first term is trivially estimated.
    Otherwise $\bou E\cap \co1\infty$ is nonempty and must have a supremum point, that we denote by $t_1$ (we may assume $t_1<\infty$, so that the right-hand side in \cref{eq:no_fantasy} is finite). It holds
    \begin{equation*}
        \int_{\co1\infty\cap E} t^\gamma\de t
        \le 
        \int_1^{t_1} t^\gamma\de t
        \le t_1^{\gamma}\abs{t_1-1}
        \le \int_{\bou E} t^{\gamma}\abs{1-t}\de\Haus^0(t) \comma
    \end{equation*}
    hence we have successfully controlled the first term.
    
    Let us now move our attention to the second term. First notice that its value is a priori bounded by~$1$. If $\bou E\cap\cc{\frac14}{\frac34}\not=\emptyset$, then the right-hand side of \cref{eq:no_fantasy} is at least $C_\gamma \left(\frac{1}{4}\right)^{\gamma+1}$ and therefore we have the desired estimate (choosing $C_\gamma$ appropriately). 
    
    Thus we can assume $\bou E\cap\cc{\frac14}{\frac34}=\emptyset$.
    If $\cc{\frac14}{\frac34}\setminus E\not = \emptyset$, then (since $E$ has no boundary in that interval) it follows $\cc{\frac14}{\frac34}\cap E = \emptyset$ and in particular the right-hand side of \cref{eq:no_fantasy} is larger than 
    \begin{equation*}
        C_\gamma\int_{\frac14}^{\frac34}t^\gamma\de t\comma
    \end{equation*}
which yields the desired estimate.
    
    Thus we can assume $\cc{\frac14}{\frac34}\subseteq E$. If $\bou E\cap\cc{\frac14}{1}=\emptyset$, then $\cc{\frac12}{1}\subseteq E$ and therefore there is nothing to prove. Otherwise, let $t_0$ be the infimum of $\bou E\cap\cc{\frac14}{1}$.
    It holds
    \begin{equation*}
        \int_{\cc{\frac12}{1}\setminus E} t^\gamma\de t
        \le
        \int_{t_0}^1 t^{\gamma}\de t
        \le 4^\gamma t_0^{\gamma}\abs{1-t_0}
        \le 4^\gamma \int_{\bou E} t^{\gamma}\abs{1-t}\de\Haus^0(t)\comma
    \end{equation*}
    and this concludes the proof.
\end{proof}

To perform the polar slicing we will need the following technical lemma.

\begin{lemma}\label{lem:easy_polar_formulas}
    Let $\Sigma$ be an open cone and let $E\subseteq\Sigma$ be a measurable set such that $\Per(E,\Omega)<\infty$ for any $\Omega\compactsubset\Sigma$.
    For any $\theta\in \S^{n-1}\cap\Sigma$, let us define
    \begin{equation*}
        E_{\theta}\defeq\{t\ge 0:\ t\theta\in E\} \fullstop
    \end{equation*}
    Then, for any $\eta\in L^1(E)$, we have
    \begin{equation*}
        \int_E \eta 
        = \int_{\S^{n-1}\cap\Sigma}\de\Haus^{n-1}(\theta)
        \int_{E_{\theta}} t^{n-1}\eta(t\theta)\de t \fullstop
    \end{equation*}

    Moreover, for $\Haus^{n-1}$-almost every $\theta\in\S^{n-1}\cap\Sigma$, $E_{\theta}\subseteq\R$ is a $1$-dimensional set of locally finite perimeter such that the Vol'pert property $\bou E_{\theta}\cap\{t>0\} = \{t>0:\ t\theta\in\bou E\}$ holds. 
    Furthermore, for any nonnegative function $\eta\in L^1(\bou E,\Haus^{n-1})$, we have
    \begin{equation*}
        \int_{\bou E\cap\Sigma} \eta 
        \ge \int_{\S^{n-1}\cap\Sigma}\de\Haus^{n-1}(\theta)
        \int_{\bou E_{\theta}} t^{n-1}\eta(t\theta)\de\Haus^0(t) \fullstop
    \end{equation*}
\end{lemma}

\begin{proof}
    The proof is standard and technical, we give only a sketch.
    The first part of the statement follows from the coarea formula (\cite[Theorem 13.1]{maggi2012}), whereas the second part can be shown applying \cite[Theorem 3.107]{ambrosio-fusco-pallara} and the area formula for rectifiable sets (\cite[Theorem 11.6]{maggi2012}).
\end{proof}

Using the previous two results, we can now give the:

\begin{proof}[Proof of \cref{prop:close_to_translated_ball}]
    The strategy of the proof is to perform a polar slicing of the set $E$ and apply \cref{lem:one_dim_boundary_to_inside} on each slice.
    
    Using \cref{lem:easy_polar_formulas}, we obtain
    \begin{align*}
        \int_{\bou E\cap\Sigma} \abs{\abs{x-x_0}-1}w(x)\de\Haus^{n-1}(x)
        \ge \int_{\S^{n-1}\cap\Sigma}\, w(\theta)\de\Haus^{n-1}(\theta)
        \int_{\bou E_{\theta}} t^{D-1}\abs{\abs{t\theta-x_0}-1}\de\Haus^0(t) \fullstop
    \end{align*}
    Keeping in mind that $x_0$ is a small vector, we have that the set $\{ t>0:\ \abs{t\theta-x_0}<1\}$ is an open segment $\oo0{t(\theta)}$, with $t(\theta)$ close to $1$, for every $\theta\in\S^{n-1}$. In addition, for any $\theta\in\S^{n-1}$ and $t>0$, we have $\abs{\abs{t\theta-x_0}-1}\gtrsim \abs{t-t(\theta)}$, where the hidden constant is purely geometric. 
    Hence we obtain
    \begin{equation}\label{eq:tmp098}
    \begin{split}
        \int_{\S^{n-1}\cap\Sigma}w(\theta)\de\Haus^{n-1}(\theta)
        \int_{\bou E_{\theta}} t^{D-1}\abs{t-t(\theta)}\de\Haus^0(t)
        \lesssim
        \int_{\bou E\cap\Sigma} \abs{\abs{x-x_0}-1}w(x)\de\Haus^{n-1}(x) \fullstop
    \end{split}
    \end{equation}
    
    Applying \cref{lem:easy_polar_formulas} and the relative isoperimetric inequality, \cref{cor:rel_isop_ball}, to the set $F\defeq\big(B_{\frac12}\cap\Sigma\big)\setminus E$, we deduce
    \begin{equation}
    \begin{split}\label{eq:tmp2324}
        \int_{\S^{n-1}\cap\Sigma} w(\theta)\de\Haus^{n-1}(\theta) & \int_{\cc0{\frac12}\setminus E_{\theta}} t^{D-1}\de t
        = 
        w\big(\big(B_\frac12\cap\Sigma\big)\setminus E\big) \\
        &\lesssim
        \int_{\bou E\cap B_{\frac12}\cap\Sigma} w\de\Haus^{n-1}
        \lesssim \int_{\bou E\cap\Sigma}\abs{\abs{x-x_0}-1}w(x)\de\Haus^{n-1}(x) \fullstop
    \end{split}
    \end{equation}
Here we used that $\abs{\abs{x-x_0}-1}\gtrsim 1$ in $B_{\frac12}$.
    
    Using \cref{lem:easy_polar_formulas}, the estimates \cref{eq:tmp098,eq:tmp2324}, and \cref{lem:one_dim_boundary_to_inside}, we conclude
    \begin{equation*}
        w\big(E\triangle (B_1(x_0)\cap\Sigma)\big)
        =
        \int_{\S^{n-1}\cap\Sigma}w(\theta)\de\Haus^{n-1}(\theta)
        \int_{E_{\theta}\triangle\cc{0}{t(\theta)}} t^{D-1}\de t
        \lesssim
        \int_{\bou E\cap\Sigma}\abs{\abs{x-x_0}-1}w(x)\de\Haus^{n-1}(x) \fullstop
    \end{equation*}
\end{proof}

\section{Weighted trace and Sobolev-Poincar\'e inequalities}\label{sec:FMP}

In this section we establish a trace inequality and a Sobolev--Poincar\'e inequality in the weighted setting and within a cone, as well as a uniform bound on the Cheeger constant for sets that have small weighted isoperimetric deficit.
Essentially, we repeat the arguments of \cite[Section 3]{FMP}, taking care of the presence of the weight $w$ and of the cone $\Sigma$. The only additional result contained here is the Sobolev-Poincar\'e inequality \cref{eq:poincare_ineq}. Both the statements and the proofs in this section are simple adaptations of the analogous ones in \cite{FMP}.
Unexpectedly, the constants in the statements of \cref{lem:removal} and \cref{thm:nice_subset_with_poincare} depend only on $D=n+\alpha$ and not on $\Sigma$ or $w$.

\begin{remark}
In this section, we never use directly the homogeneity or the concavity of the weight; all the results remain true for any weight such that the weighted isoperimetric inequality holds.
\end{remark}

Given a set of finite $w$-perimeter $E\subseteq\Sigma$ with $0<w(E)<\infty$, we define the Cheeger constant as
\begin{equation}\label{eq:cheeger_constant}
    \tau(E)\defeq \inf\left\{\frac{\Per_w(F)}{\Haus^{n-1}_w(\bou F\cap \bou E)}:\;F\subseteq E,\;0<w(F)\leq\frac{w(E)}{2}\right\}\,.
\end{equation}
It follows from the definition of $\tau(E)$ that, as long as $\tau(E)>1$, we have the following relative isoperimetric inequality
\begin{equation}\label{eq:rel_isop_general}
c_*w(F)^{\frac{D-1}D}\le \Per_w(F) \le \frac{\tau(E)}{\tau(E)-1}\,\Haus^{n-1}_w(\bou F\cap E)
\end{equation}
for any $F\subseteq E$ with $w(F)\leq \frac12 w(E)$, where $c_*\defeq\Per_w(B_1\cap \Sigma)/w(B_1\cap \Sigma)^{\frac{D-1}{D}}=Dw(B_1\cap\Sigma)^{\frac1D}$ is the isoperimetric constant that comes from \cref{isop-ineq}, and  $D\defeq n+\alpha$.

\begin{remark}
Notice that when $\tau(E)-1$ is very small, \cref{eq:rel_isop_general} is not very useful (as the constant in the right-hand side explodes). We will see that also the trace inequality \cref{eq:trace_ineq} and the Sobolev-Poincar\'e inequality \cref{eq:poincare_ineq} exhibit a similar behavior. Since we want to apply these inequalities on sets with small weighted isoperimetric deficit, it is crucial to show that, if $\delta_w(E)$ is sufficiently small, then (up to slightly modifying the set $E$) the value $\tau(E)-1$ is bounded away from $0$ by a constant that does not depend on $E$ (see \cite[Section 1.6]{FMP} for an explanation of why it is necessary to modify the set $E$). This is exactly the statement of \cref{thm:nice_subset_with_poincare}.
\end{remark}

The following trace inequality is the analogue of \cite[Lemma 3.1]{FMP}. 
In addition to the trace inequality, we prove also a Sobolev-Poincar\'e inequality.
See \cref{sec:notation} for the definition of $E^{(1)}$.

\begin{lemma}\label{lem:trace_poincare}
Consider $n,\,\alpha,\,\Sigma,\,w$ satisfying \cref{eq:assumptions}.
For every function $f\in BV_w(\R^n)\cap L^\infty(\R^n)$ and for every set of finite $w$-perimeter $E\subseteq\Sigma$ with $w(E)<\infty$, there is a constant\footnote{The constant $c$, as it is clear from the proof, can be chosen to be the \emph{median} value of $f$.} $c\in\R$ such that
\begin{equation}\label{eq:trace_ineq}
\int_{E^{(1)}} w(x)\de \abs{\diff f}(x) \geq (\tau(E)-1)\int_{\bou E\cap\Sigma} \tr_E(|f-c|)w(x)\de\Haus^{n-1}(x)\comma
\end{equation}
and also
\begin{equation}\label{eq:poincare_ineq}
    \int_{E^{(1)}} w(x)\de \abs{\diff f}(x) 
    \ge D \big(1-\tau(E)^{-1}\big) 
    \left(\int_E \abs{f-c}^{\frac{D}{D-1}}w(x)\de x\right)^{\frac{D-1}{D}} \fullstop
\end{equation}
\end{lemma}

\begin{proof}
For every $t\in \R$ we consider the set $F_t\defeq E\cap \{f>t\}$. There exists $c\in \R$ such that $w(F_t)\leq w(E)/2$ for every $t\geq c$, and $w(E\setminus F_t)\leq w(E)/2$ for every $t<c$. We set moreover $g\defeq (f-c)_+$, where $(A)_+$ denotes the positive part of $A$, and $G_s\defeq E\cap \{g>s\}.$

Using the coarea formula \cite[Theorem 13.1]{maggi2012}, we have that
\begin{equation}\label{ineq-0}
\int_{E^{(1)}} w(x)\de \abs{\diff g}(x)=\int_0^{\infty}ds\int_{E^{(1)}\cap\bou\{g>s\}}w(x)\de\Haus^{n-1}(x) \fullstop
\end{equation}
Moreover, by the definition of $\tau(E)$, we have
\begin{equation}\label{ineq-1}
\begin{split}
\int_{E^{(1)}\cap \bou\{g>s\}}w(x)\de\Haus^{n-1}(x)&=\int_{E^{(1)}\cap \bou G_s} w(x)\de\Haus^{n-1}(x)\\
&=\int_{\bou G_s\cap\Sigma} w(x)\de\Haus^{n-1}(x)-\int_{\bou E\cap \bou G_s\cap\Sigma} w(x)\de\Haus^{n-1}(x)\\
&\geq (\tau(E)-1)\int_{\bou E\cap \bou G_s\cap\Sigma} w(x)\de\Haus^{n-1}(x) \fullstop
\end{split}
\end{equation}
Now, by Fubini, we have
$$\int_{\bou E\cap\Sigma}\tr_E(g)w(x)\de\Haus^{n-1}(x)=\int_0^{\infty}\de s\int_{\bou E\cap \{\tr_E(g)>s\}\cap\Sigma}w(x)\de\Haus^{n-1}(x).$$
Using that, up to $\Haus^{n-1}$-null sets, $\bou E\cap\{\tr_E(g)>s\}\subseteq \bou E\cap \bou G_s$ (the proof of this fact is contained in the proof of \cite[Lemma 3.1]{FMP}), we deduce that
\begin{equation}\label{ineq-2}
\int_{\bou E\cap\Sigma}\tr_E(g)w(x)\de\Haus^{n-1}(x)\le \int_0^{\infty}\de s\int_{\bou E\cap \bou G_s\cap\Sigma}w(x)\de\Haus^{n-1}(x) \fullstop
\end{equation}
Combining together \cref{ineq-1} and \cref{ineq-2} into \cref{ineq-0}, we get
\begin{equation*}
\begin{split}
\int_{E^{(1)}} w(x)\de\abs{\diff g}(x) 
&\geq (\tau(E)-1)\int_0^{\infty}\de s\int_{\bou E\cap \bou G_s\cap\Sigma} w(x) \de\Haus^{n-1}(x)\\
&\geq (\tau(E)-1)\int_{\bou E\cap\Sigma}\tr_E(g)\,w(x)\de \Haus^{n-1}(x) \fullstop
\end{split}
\end{equation*}
Now, we repeat the above argument with $(f-c)_-$ in place of $(f-c)_+$ and, using the linearity of the trace operator and the fact that $(f-c)_+ + (f-c)_-=|f-c|$, we deduce that
\begin{align*}
&\int_{E^{(1)}} w(x) \de \abs{\diff((f-c)_+)}(x) + \int_{E^{(1)}} w(x) \de \abs{\diff((f-c)_-)}(x)\\
&\quad\quad\quad\quad
\ge (\tau(E)-1)\int_{\bou E\cap\Sigma}\tr_E(|f-c|)w(x)\de\Haus^{n-1}(x) \fullstop
\end{align*}
To conclude the proof of \cref{eq:trace_ineq} it is enough to show that, for any open set $\Omega$ in $\R^n$
\begin{equation*}
\int_{\Omega} w(x) \de \abs{\diff f}(x)
\ge
\int_{\Omega} w(x) \de \abs{\diff((f-c)_+)}(x) + \int_{\Omega} w(x) \de \abs{\diff((f-c)_-)}(x) \fullstop
\end{equation*}
This fact can be seen exactly as in \cite{FMP}, by approximating $f$ with smooth functions in the weighted BV-norm, and using the lower semicontinuity of the weighted total variation.

Let us now move our attention to the Sobolev-Poincar\'e inequality. Repeating the argument of \cref{ineq-0,ineq-1} and applying the weighted isoperimetric inequality, we can show
\begin{align*}
    \int_{E^{(1)}} w(x)\de\abs{\diff g}(x) 
    \ge (1-\tau(E)^{-1}) \int_0^{\infty}\de s\int_{\bou G_s\cap\Sigma} w(x)\de\Haus^{n-1}(x) \ge D(1-\tau(E)^{-1}) \int_0^{\infty}w(G_s)^{\frac{D-1}{D}} \de s 
    \fullstop
\end{align*}
The estimate \cref{eq:poincare_ineq} follows from the last one as described in \cite[Theorem A.25]{ambrosio-carlotto-massacesi2018}.
\end{proof}

We define now the family of sets (that depends on a set $E$ that will always be clear from the context)
\begin{equation*}
    \Gamma_{\lambda} \defeq \left\{F\subseteq E:\:0<w(F)\leq \frac{w(E)}{2},\;\Per_w(F)\leq \lambda \Haus_w^{n-1}(\bou F\cap \bou E)\right\} \fullstop
\end{equation*}
The following lemma is the analogue \cite[Lemma 3.2]{FMP}.
\begin{lemma}\label{maximal}
Consider $n,\,\alpha,\,\Sigma,\,w$ satisfying \cref{eq:assumptions}.
Let $E\subseteq\Sigma$ be a set of finite $w$-perimeter with $0<w(E)<\infty$, and let $\lambda>1$. If the family $\Gamma_\lambda$ is not empty, then it admits a maximal element with respect to the order relation given by set inclusion up to sets of measure zero.
\end{lemma}
\begin{proof}
We define the increasing sequence $(F_i)_{i\in\N}$ of sets in $\Gamma_\lambda$ in the following way. Let $F_1$ be any element of $\Gamma_\lambda$ and, once $F_i$ has been defined, we consider
$$\Gamma_\lambda(i)=\{F\in \Gamma_\lambda:\,F_i \subseteq F\}.$$
Now, let $F_{i+1}$ be an element of $\Gamma_\lambda(i)$ which satisfies
$$w(F_{i+1})\geq \frac{w(F_i)+s_i}{2},\quad \mbox{where}\;s_i=\sup_{F\in \Gamma_\lambda(i)} w(F).
$$
Since $F_i$ is an increasing sequence of sets it admits a limit, we call it $F_\infty$. In what follows we show that $F_\infty \in \Gamma_\lambda$ and $F_\infty$ is a maximal element in $\Gamma_\lambda$.

First, we observe that $w(F_\infty)=\sup_{i\in \N}w(F_i)\leq w(E)/2$. Moreover, by lower semicontinuity of the weighted perimeter, we have
$$ \Per_w(F_\infty)\leq \liminf_{i\to\infty} \Per_w(F_i)\leq \lambda \liminf_{i\to\infty}\H_w^{n-1}(\bou F_i \cap \bou E).$$
Since $F_i \subseteq F_{i+1}\subseteq F_\infty\subseteq E$, we have (up to $\Haus^{n-1}$-negligible sets)
\begin{equation*}
    (\bou F_i \cap \bou E)\subseteq (\bou F_{i+1} \cap \bou E)\subseteq (\bou F_\infty \cap \bou E) \comma
\end{equation*}
therefore $F_\infty \in \Gamma_\lambda$. In order to show the maximality of $F_\infty$, we consider another subset $H\subseteq E$ such that $H\cap F_\infty=\emptyset$ and $H\cup F_\infty \in \Gamma_\lambda$. By construction $F_\infty \cup H \in \Gamma_\lambda(i)$, so that for every $i\in\N$
$$s_i\geq w(F_\infty \cup H)\geq w(F_{i+1})+w(H)\geq \frac{w(F_i)+s_i}{2}+w(H),$$
that is, $w(H)\leq (s_i-w(F_i))/2$. Since $s_i-w(F_i)\leq 2w(F_{i+1}\setminus F_i)\rightarrow 0$ as $i\rightarrow \infty$, we have deduced that $w(H)=0$, which gives the maximality of $F_\infty$.
\end{proof}

Recall that we denote the isoperimetric deficit by
$$\delta_w(E)\defeq\frac{\Per_w(E)}{D w(E)^{\frac{D-1}{D}}}-1,$$
where $D=n+\alpha$ is the isoperimetric constant (since we are assuming $w(B_1\cap\Sigma)=1$).

We want to show that if $E$ is almost optimal, then any subset $F$ of $E$ which makes $\tau(E)$ small enough has small volume. In order to do that, following \cite{FMP}, we introduce the strictly concave function $\Psi:[0,1]\rightarrow [0, 2^{1/D}-1]$ given by
$$\Psi(t)\defeq t^{\frac{D-1}{D}}+(1-t)^{\frac{D-1}{D}}-1.$$
We observe that
\begin{equation}\label{psi}
\Psi(t)\geq (2-2^{\frac{D-1}{D}})t^{\frac{D-1}{D}},\quad \mbox{for}\quad t\in \cc0{\tfrac{1}{2}}
\end{equation}
and we set
\begin{equation}\label{kD}
k(D)\defeq \frac{2-2^{\frac{D-1}{D}}}{3},
\end{equation}
so that $\Psi(t)\geq 3 k(D)t^{\frac{D-1}{D}}$ for $t\in\cc0{\frac12}$.

The following lemma is the analogue of \cite[Lemma 3.3]{FMP}.
\begin{lemma}\label{lem:removal}
Consider $n,\,\alpha,\,\Sigma,\,w$ satisfying \cref{eq:assumptions}.
Let $E, F$ be two sets of finite $w$-perimeter, with $F\subseteq E\subseteq \Sigma$ such that
\begin{equation}\label{hp-removal}
0<w(F)<\frac{w(E)}{2}<\infty \quad \mbox{and}\quad \Per_w(F)\leq(1+k(D))\mathcal H_w^{n-1}(\bou E\cap \bou F) \fullstop
\end{equation}
Then we have:
\begin{enumerate}[label=(\roman*)] \setlength\itemsep{0.6em}
\item \label{it:removal_volume} $\displaystyle w(F)\leq \left(\frac{\delta_w(E)}{k(D)}\right)^{\frac{D}{D-1}}w(E)$;
\item  \label{it:removal_perimeter} $\displaystyle \Per_w(E\setminus F)\leq \Per_w(E)$;
\item  \label{it:removal_deficit} If $\delta_w(E)\leq k(D)$, then $\displaystyle\delta_w(E\setminus F)\leq \frac{3}{k(D)}\delta_w(E)$.
\end{enumerate}
\end{lemma}
\begin{proof}
Using the second inequality in \cref{hp-removal}, we have that
\begin{align}
\Per_w(E) 
&= \Per_w(E\setminus F)+\Per_w(F)-2\Haus^{n-1}_w(\bou F\cap E^{(1)})\nonumber \\
&=\Per_w(E\setminus F)+\Per_w(F)-2\left(\Per_w(F)-\Haus^{n-1}_w(\bou F\cap \bou E)\right)\nonumber \\
&\geq \Per_w(E\setminus F)+\Per_w(F)-2k(D)\Haus^{n-1}_w(\bou F\cap \bou E)\nonumber \\
&\geq \Per_w(E\setminus F)+\Per_w(F)\left(1-2k(D)\right) \label{1-removal}\\
&\geq D\, w(E\setminus F)^{\frac{D-1}{D}} + \left(1-2k(D)\right) D\, w(F)^{\frac{D-1}{D}}
\comma \label{2-removal}
\end{align}
where in the last estimate we have applied the weighted isoperimetric inequality to the sets $E\setminus F$ and $F$.

We set now $t\defeq w(F)/w(E)$, so that $w(E\setminus F)/w(E)=1-t$ and, by the first assumption in \cref{hp-removal} we have $t\leq 1/2$. Dividing by $D\,w(E)^{\frac{D-1}{D}}$ in \cref{2-removal}, we get
\begin{equation*}
\delta_w(E)=\frac{\Per_w(E)}{D\,w(E)^{\frac{D-1}{D}}}-1\geq(1-t)^{\frac{D-1}{D}} + \left(1-2k(D)\right) t^{\frac{D-1}{D}} -1.
\end{equation*}
Now we use the definitions of $\Psi$ and $k(D)$ and inequality \cref{psi}, to deduce
\begin{equation}\label{3-removal}
\delta_w(E)\geq \Psi(t)-2k(D) t^{\frac{D-1}{D}}\geq k(D) t^{\frac{D-1}{D}} =
k(D)\Big(\frac{w(F)}{w(E)}\Big)^{\frac{D-1}{D}}\comma
\end{equation}
that is equivalent to \cref{it:removal_volume}.

The estimate \cref{it:removal_perimeter} follows from \cref{1-removal}, using that $1-2k(D)\geq 0$.

It remains to show \cref{it:removal_deficit}.
First we observe that \cref{3-removal} and the assumption $\delta_w(E)\leq k(D)$ imply that
$$t\leq \left(\frac{\delta_w(E)}{k(D)}\right)^{\frac{D}{D-1}}\leq \frac{\delta_w(E)}{k(D)}.$$
Therefore we have that
\begin{equation*}
\begin{split}
\delta_w(E\setminus F)&=\frac{\Per_w(E\setminus F)}{D\, w(E\setminus F)^{\frac{D-1}{D}}}-1
=\frac{\Per_w(E\setminus F)}{D\ (1-t)^{\frac{D-1}{D}}w(E)^{\frac{D-1}{D}}}-1\\
&\leq \frac{\Per_w(E\setminus F)}{D\, w(E)^{\frac{D-1}{D}}}(1+2t)-1
=\delta_w(E)+2t\left(\delta_w(E)+1\right)\leq \frac{3}{k(D)}\delta_w(E) \comma
\end{split}
\end{equation*}
as wanted.
\end{proof}

Finally, the following theorem is the analogue of \cite[Theorem 3.4]{FMP}. It states that if $E$ has small isoperimetric deficit, then there exists a subset $G$ of $E$ which also has small deficit and, more importantly, such that $\tau(G)-1$ is bounded below away from zero. The idea of the proof consists in cutting away from $E$ the maximal critical set (whose existence is established in \cref{maximal}) and using the estimates of \cref{lem:removal}.
\begin{theorem}\label{thm:nice_subset_with_poincare}
Consider $n,\,\alpha,\,\Sigma,\,w$ satisfying \cref{eq:assumptions}.
Let $E\subseteq\Sigma$ be a set of finite weighted perimeter and suppose that $\delta_w(E)\leq k^2(D)/8$, with $k(D)$ given by \cref{kD}.

Then, there exists a set $G\subseteq E$ with finite $w$-perimeter which satisfies the following estimates:
\begin{enumerate}[label=(\alph*)] \setlength\itemsep{0.6em}
\item \label{it:fmp_volume} $\displaystyle w(E\setminus G)\leq \frac{\delta_w(E)}{k(D)}w(E)$;
\item \label{it:fmp_deficit} $\displaystyle \delta_w(G)\leq \frac{3}{k(D)}\delta_w(E)$;
\item \label{it:fmp_poincare} $\displaystyle \tau(G)\geq 1+k(D)$.
\end{enumerate}
\end{theorem}
\begin{proof}
If $\tau(E)\ge 1 + k(D)$, then we can choose $G\defeq E$ and there is nothing to prove.
Otherwise, let $F_\infty$ be the maximal critical set given in \cref{maximal} with $\lambda=1+k(D)$ (notice that $\Gamma_\lambda\not=\emptyset$ because $\tau(E)<1+k(D)$). We define the set $G$ as $G\defeq E\setminus F_\infty$. Since $F_\infty\in \Gamma_\lambda$, then we can apply \cref{lem:removal} to $F=F_\infty$ and we deduce the estimates \cref{it:fmp_volume,it:fmp_deficit}.

It remains to prove \cref{it:fmp_poincare}. In order to do that, we argue by contradiction, that is we assume that $\tau(G)<1+k(D)=\lambda$. By definition of $\tau(G)$ we have that there exists a set $H\subseteq G$ with $0\leq w(H)\leq w(G)/2$ and such that
$$\Per_w(H)<\lambda \Haus^{n-1}_w(\bou H\cap \bou G).$$
We aim to show that $F_\infty\cup H\in\Gamma_\lambda$, which gives a contradiction to the maximality of $F_\infty$. Thus, we will have completed the proof once we have checked that
\begin{equation}\label{contra}
0\leq w(F_\infty \cup H)\leq \frac{w(E)}{2} \quad \mbox{and}\quad \Per_w(F_\infty \cup H)\leq \lambda\Haus^{n-1}_w(\bou(F_\infty\cup H)\cap \bou E).
\end{equation}
Applying \cref{it:fmp_deficit}, the first estimate follows from \cref{lem:removal}. Indeed applying \cref{lem:removal} to $H\subseteq G$, we have that
\begin{equation*}
    w(H) \le 
    \frac{\delta_w(G)}{k(D)}w(G)\leq 3\frac{\delta_w(E)}{k^2(D)}w(E) \fullstop
\end{equation*}
Moreover, using that $F_\infty$ and $H$ are disjoint and applying \cref{lem:removal} to $F_\infty \subseteq E$, we deduce
\begin{equation*}
\begin{split}
w(F_\infty \cup H) = w(F_\infty)+w(H)\leq \frac{\delta_w(E)}{k(D)}w(E)+3\frac{\delta_w(E)}{k^2(D)}w(E)
\leq 4\frac{\delta_w(E)}{k^2(D)}w(E)\leq \frac{w(E)}{2} \comma
\end{split}
\end{equation*}
where in the last inequality we have used the assumption $\delta_w(E)\leq k^2(D)/8$.

It remains to prove the second estimate of \cref{contra}. We start by observing that
$$\Per_w(F_\infty \cup H)=\Haus^{n-1}_w(\bou(F_\infty \cup H)\cap E^{(1)})+
\Haus^{n-1}_w(\bou(F_\infty \cup H)\cap \bou E).$$
Since $\lambda=k(D)+1$, in order to conclude we have to show that
\begin{equation}\label{final}
\Haus^{n-1}_w(\bou(F_\infty \cup H)\cap E^{(1)})\leq k(D)
\Haus^{n-1}_w(\bou(F_\infty \cup H)\cap \bou E).
\end{equation}
First, we write
\begin{equation}\label{6}
\Haus^{n-1}_w(\bou(F_\infty \cup H)\cap E^{(1)})=\Haus^{n-1}_w((\bou F_\infty \setminus \bou H))\cap E^{(1)})+\Haus^{n-1}_w((\bou H\setminus \bou F_\infty)\cap E^{(1)})
\end{equation}
We now estimate the second term on the right-hand side of \cref{6}. Since (up to $\Haus^{n-1}$-negligible sets) it holds $(\bou H\setminus \bou F_\infty)\cap E^{(1)}\subseteq G^{(1)}$, we have
\begin{equation}\label{7}
\begin{split}
\Haus^{n-1}_w((\bou H\setminus \bou F_\infty)\cap E^{(1)})&\leq \Haus^{n-1}_w(\bou H\cap G^{(1)}) \\
&=\Per_w(H)-\Haus^{n-1}_w(\bou H\cap \bou G)\\
&\leq k(D) \Haus^{n-1}_w(\bou H\cap \bou G),
\end{split}
\end{equation}
where in the last estimate we have used $\Per_w(H)\leq \lambda \Haus^{n-1}_w(\bou H\cap \bou G)$.

Combining together \cref{6} and \cref{7}, and using that $F_\infty \in \Gamma_\lambda$, we deduce that
\begin{equation*}
\begin{split}
\Haus^{n-1}_w(\bou(F_\infty \cup H)\cap E^{(1)})&\leq \Haus^{n-1}_w((\bou F_\infty \setminus \bou H)\cap E^{(1)})+k(D)\Haus^{n-1}_w(\bou H\cap \bou G)\\
&=\Haus^{n-1}_w((\bou F_\infty \setminus \bou H)\cap E^{(1)})\\
&\hspace{1em}+k(D)\Big[\Haus^{n-1}_w((\bou H\cap \bou G)\cap E^{(1)})
+\Haus^{n-1}_w((\bou H\cap \bou G)\cap \bou E)\Big] \\
&\leq \Haus^{n-1}_w(\bou F_\infty \cap E^{(1)}) + k(D) \Haus^{n-1}_w((\bou H\cap \bou G)\cap \bou E)\\
&\leq k(D) \Big[\Haus^{n-1}_w(\bou F_\infty \cap \bou E) + \Haus^{n-1}_w((\bou H\cap \bou G)\cap \bou E)\Big]\\
&\leq k(D) \Haus^{n-1}_w(\bou (F_\infty \cup H)\cap \bou E) \fullstop
\end{split}
\end{equation*}
We have established \cref{final} and therefore the proof is concluded.
\end{proof}

\begin{corollary}[Relative isoperimetric inequality in the ball]\label{cor:rel_isop_ball}
    For any $F\subseteq B_1\cap \Sigma$ such that $w(F)\leq \frac12 w(B_1\cap \Sigma)$, it holds
    $$
    w(F)\lesssim \Haus^{n-1}_w(\bou F\cap B_1\cap \Sigma) \,.
    $$
\end{corollary}
\begin{proof}
    Since $\delta_w(B_1\cap\Sigma) = 0$, \cref{thm:nice_subset_with_poincare} (with $E=G=B_1\cap\Sigma$) tells us that $\tau(B_1\cap\Sigma)\ge 1 + k(D)$. Then, the statement follows from \cref{eq:rel_isop_general}.
\end{proof}

\section{Proof of the main result}\label{sec:main_proof}

We prove three increasingly stronger statements: the characterization of optimal sets, \cref{prop:uniqueness}, a nonquantitative stability for almost-optimal sets, \cref{lem:nonquantitative}, and finally the quantitative weighted isoperimetric inequality, \cref{thm:main}. 
We prove them separately because we use the characterization of optimal sets in the proof of the nonquantitative stability and we apply the nonquantitative stability in the proof of the quantitative weighted isoperimetric inequality.

\begin{proof}[Proof of \cref{prop:uniqueness}]
    Let $E\subseteq\Sigma$ be an optimal set for the weighted isoperimetric inequality. Without loss of generality we can assume $w(E)=w(B_1\cap\Sigma)=1$.
    Thanks to \cref{thm:nice_subset_with_poincare} and \cref{lem:trace_poincare}, we know that $E$ satisfies a weighted Poincar\'e inequality (since it must hold $G=E$ in the statement of \cref{thm:nice_subset_with_poincare} because $\delta_w(E)=0$).
    Let $\varphi:\R^n\to\R^n$ be the convex map described in \cref{prop:weighted_coupling_properties} (notice that $E$ is $w$-indecomposable because $\tau(E)>1$).
    From \cref{eq:hessian_control}, it follows that there exists $x_0\in\R^n$ such that $\nabla\varphi(x)-x=-x_0$ for any $x\in E$. Hence, since $\nabla\varphi(E)=B_1\cap\Sigma$, it holds $E = x_0 + B_1\cap\Sigma$.
    To conclude, notice that $\Per_w(E)=\Per_w(B_1\cap\Sigma)$ implies $x_0+\partial\Sigma\subseteq\partial\Sigma$ and therefore $x_0\in\R^k\times\{0_{\R^{n-k}}\}$.
\end{proof}

Using the characterization of optimal sets, we next show the following.

\begin{lemma}[Nonquantitative stability]\label{lem:nonquantitative}
    Consider $n,\,\alpha,\,\Sigma,\,w$ satisfying \cref{eq:assumptions} and let $(E_i)_{i\in\N}\subseteq\Sigma$ be a sequence of sets of finite $w$-perimeter such that $E_i\subseteq\Sigma$, $w(E_i)=1$, and $\delta_w(E_i)\to 0$.
    Then $w(E_i\triangle (B_1(x_i)\cap\Sigma))\to 0$ for an appropriate choice of $x_i\in\R^k\times\{0_{\R^{n-k}}\}$.
\end{lemma}

\begin{proof}
We assume that $\Sigma$ contains no lines, that is $k=0$. It is easy to adapt the proof to handle the case $k>0$.
We call a \emph{minor perturbation} of the sequence, any replacement of $E_i$ with $E'_i$, with $w(E_i\triangle E'_i)\to 0$ and $\delta_w(E_i')\to 0$. We will not change the naming of the sets when performing minor perturbations. 

Applying \cref{thm:nice_subset_with_poincare}, up to a minor perturbation, we can assume $\tau(E_i) \ge 1 + k(D)$. Hence, thanks to \cref{lem:trace_poincare}, the sets $E_i$ enjoy \emph{nontrivial} Poincar\'e and trace inequalities. Let us denote with $\varphi_i:\R^n\to\R$ the convex function described in \cref{prop:weighted_coupling_properties} relative to $E_i$ (notice that $E_i$ is $w$-indecomposable because $\tau(E_i)>1$).
From the estimate \cref{eq:hessian_control}, thanks to the trace and Poincar\'e inequalities, it follows that there exists a sequence of points $(x_i)_{i\in\N}\subseteq\R^n$ such that
\begin{align}
    &\int_{\bou E_i\cap\Sigma} \abs{\nabla\varphi_i - (x-x_i)}\, w(x)\de\Haus^{n-1}(x) 
    \longrightarrow 0 \quad\text{as $i\to\infty$}\comma \label{eq:main_proof_boundary_nabla_i}\\
    &\int_{E_i} \abs{\nabla\varphi_i-(x-x_i)}\, w(x)\de x 
    \longrightarrow 0 \quad\text{as $i\to\infty$} \label{eq:main_proof_inside_nabla_i}\fullstop
\end{align}
From \cref{eq:main_proof_inside_nabla_i} we deduce $w(E_i\setminus B_2(x_i))\to 0$, and therefore there is $2<r<3$ such that $E_i\to E_i\cap B_r(x_i)$ is a minor perturbation. Hence we can assume $E_i\subseteq B_3(x_i)$. Repeating the argument that led to \cref{eq:main_proof_boundary_nabla_i,eq:main_proof_inside_nabla_i} for the original $E_i$, we can also assume that \cref{eq:main_proof_boundary_nabla_i,eq:main_proof_inside_nabla_i} hold.

Since $\nabla\varphi_i(E_i)=B_1\cap\Sigma$ and $\norm{\nabla^2\varphi_i}_{\infty}$ is uniformly bounded, the area formula implies that $\inf_i\abs{E_i} > 0$.
Combining \cref{eq:main_proof_boundary_nabla_i} with \cref{eq:boundary_control} we deduce
\begin{equation*}
    \int_{\bou E_i\cap\Sigma}d\big(x-x_i, \partial B_1\cap\Sigma\big)\,w(x)\de\Haus^{n-1}(x) \longrightarrow 0\fullstop
\end{equation*}
Because of the concavity of $w^{\frac1\alpha}$, it holds $\inf_{x\in\Sigma} \frac{w(x)}{d(x,\partial\Sigma)^\alpha}>0$ and therefore the last inequality implies
\begin{equation}\label{eq:main_proof_compactness_boundary}
    \int_{\bou (E_i-x_i)}d\big(x, \partial B_1\cap\Sigma\big)\,d(x,\Sigma-x_i)^\alpha\de\Haus^{n-1}(x) \longrightarrow 0\fullstop
\end{equation}
In the next two paragraphs we will use repeatedly the compactness of sets of finite perimeter with bounded perimeter (\cite[Theorem 12.26]{maggi2012}).

If $\abs{x_i}$ stays bounded, then (up to subsequence) the sequence $E_i$ converges to a set $E_\infty$ with $w(E_{\infty})=1$ and $\delta_w(E_{\infty})=0$. Then, since $B_1\cap\Sigma$ is the unique minimizer of the weighted isoperimetric inequality (see \cref{prop:uniqueness}), we deduce $E_\infty = B_1\cap\Sigma$.

On the other hand, let us show that $\abs{x_i}\to\infty$ yields a contradiction.
Notice that $\Sigma-x_i$ subconverge (locally) to a convex cone $\Sigma'$ that contains a line (here we use $\abs{x_i}\to\infty$) and such that $\Sigma\subseteq\Sigma'$. Since $\Sigma'$ contains a line, and $\Sigma$ does not, the cone $\Sigma'$ is \emph{strictly} larger than $\Sigma$. 
Moreover the sets $E_i-x_i$ subconverge, locally in $\Sigma'$, to a nonnegligible set of locally finite\footnote{Here \emph{locally finite} has to be understood in the sense of $\Sigma'$, that is for any $\Omega\compactsubset\Sigma'$ it holds $\Per(E,\Omega)<\infty$.} perimeter $E_\infty\subseteq B_3$ with $\abs{E_\infty}>0$. Finally, thanks to the limit~\cref{eq:main_proof_compactness_boundary}, we deduce
\begin{equation}\label{eq:tmp6713}
    \bou E_{\infty} \subseteq (\partial B_1\cap\Sigma)\cup \partial\Sigma' \fullstop
\end{equation}
Since $\R^n\setminus((\partial B_1\cap\Sigma)\cup \partial\Sigma')$ has two unbounded connected components (because $\Sigma'$ is strictly larger than $\Sigma$), it follows from \cref{eq:tmp6713} that the set $E_{\infty}$ is either empty or unbounded, thus we have found a contradiction.
\end{proof}

We can finally give the:

\begin{proof}[Proof of \cref{thm:main}]
In this proof we will denote with $C$ any constant that depends on $n,\alpha, \Sigma, w$. The value of the constant can change from line to line.

Without loss of generality, we may assume that $\delta_w(E)$ is small, and that $w(E)=w(B_1\cap\Sigma)=1$.
Thanks to \cref{thm:nice_subset_with_poincare} we can also assume that the Cheeger constant $\tau(E)$ defined in \cref{eq:cheeger_constant} 
satisfies $\tau(E)-1\gtrsim 1$ (up to replacing $E$ with the set $G$ described in the statement of \cref{thm:nice_subset_with_poincare}). Let $\varphi:\R^n\to\R$ be the convex function associated to $E$ described in \cref{prop:weighted_coupling_properties} (notice that $E$ is $w$-indecomposable because $\tau(E)>1$).

Applying the trace inequality \cref{eq:trace_ineq} and the Poincar\'e inequality \cref{eq:poincare_ineq} together with \cref{eq:hessian_control}, we deduce
\begin{align}
    &\int_{\bou E\cap\Sigma} \abs{\nabla\varphi - (x - x_0)}\,w(x)\de\Haus^{n-1}(x) \lesssim \delta_w(E)^{\frac12}\comma \label{eq:main_proof_boundary_nabla}\\
    &\int_{E} \abs{\nabla\varphi - (x-x_0)}\,w(x)\de x \lesssim
    \delta_w(E)^{\frac12}\comma \label{eq:main_proof_inside_nabla}
\end{align}
for a suitable $x_0\in\R^n$. Without loss of generality we can assume that $x_0=(0_{\R^k}, \widetilde x_0)$ with $\widetilde x_0\in\R^{n-k}$. Combining \cref{eq:main_proof_boundary_nabla} with \cref{eq:boundary_control}, it follows
\begin{equation}\label{eq:tmp87123}
    \int_{\bou E\cap\Sigma} \abs{\abs{x-x_0}-1}\, w(x) \de\Haus^{n-1}(x) \lesssim \delta_w(E)^{\frac12} \fullstop
\end{equation}
Since $\delta_w(E)$ is assumed to be small, \cref{lem:nonquantitative} implies the validity of the hypotheses needed by \cref{prop:close_to_translated_ball}. Hence, recalling \cref{eq:tmp87123}, we deduce
\begin{equation}\label{eq:close_to_translated_ball}
    w\big((B_1(x_0)\cap\Sigma)\triangle E\big) \lesssim \delta_w(E)^{\frac12} \fullstop
\end{equation}

It remains only to establish that $\abs{x_0}$ is controlled by $\delta_w(E)^{\frac12}$ (as this allows to replace $x_0$ with $0_{\R^n}$ in the last inequality).
First, we prove that $\pi_{\mathcal E}(x_0)\lesssim C\delta_w(E)^{\frac12}$ (applying \cref{prop:magic_compact_set}) and then we conclude that also the component of $x_0$ along $\mathcal C$ is controlled by $\delta_w(E)^{\frac12}$ (applying \cref{prop:controlling_constant_directions}).
Let us recall the notation introduced in \cref{eq:LCE}: $\mathcal L$ is the subspace of lines contained in $\Sigma$ (that is $\R^k\times\{0_{\R^{n-k}}\}$), $\mathcal C$ is the subspace orthogonal to $\mathcal L$ such that $w$ is constant moving along $\mathcal C$, $\mathcal E$ is the orthogonal subspace to $\mathcal L\times\mathcal C$.

Notice that \cref{eq:close_to_translated_ball} implies
\begin{equation}\label{eq:contains_smaller_ball}
    w\big((B_{\frac12}\cap\Sigma)\setminus E\big) \lesssim \delta_w(E)^{\frac12} \fullstop
\end{equation}
Let $\hat\eps, \hat Q$ be the small value and the compact set described in the statement of \cref{prop:magic_compact_set}. Thanks to \cref{eq:weight_control}, we know
\begin{equation}\label{eq:tmp117}
    \int_{\hat Q\cap E} \abs{w^{\frac1\alpha}(\nabla\varphi)-w^{\frac1\alpha}(x)} \de x \le C\delta_w(E)^{\frac12}\fullstop
\end{equation}
To proceed let us assume that $\abs{x_0}<\hat\eps$ (notice that $\hat\eps$ does not depend on $E$). Recall that, thanks to \cref{lem:nonquantitative}, we can assume that $\abs{x_0}$ is arbitrarily small (and our goal is to show that it is controlled by $\delta_w(E)^{\frac12}$).
For any $x\in \hat Q$ and any $y\in B_1\cap\Sigma$, it holds
\begin{equation*}
    \abs{w^{\frac1\alpha}(x-x_0)-w^{\frac1\alpha}(y)} \le C \abs{y-(x-x_0)} \comma
\end{equation*}
where $C$ is a constant that depends on the Lipschitz constant of $w^{\frac1\alpha}$ in an $\hat\eps$-neighborhood of $\hat Q$ and on $\hat\eps$ itself. Therefore it holds
\begin{equation}\label{eq:tmp432}
    \int_{\hat Q\cap E}\abs{w^{\frac1\alpha}(\nabla\varphi)-w^{\frac1\alpha}(x-x_0)}w(x)\de x
    \le C\int_{\hat Q\cap E}\abs{\nabla\varphi - (x-x_0)}w(x) \de x
    \lesssim C \delta_w(E)^{\frac12} \comma
\end{equation}
where in the second step we used \cref{eq:main_proof_inside_nabla}.

Since $w\ge C$ on $\hat Q$, the inequalities \cref{eq:tmp432,eq:tmp117} imply
\begin{equation}\label{eq:tmp581}
    \int_{\hat Q\cap E} \abs{w^{\frac1\alpha}(x)-w^{\frac1\alpha}(x-x_0)} \de x 
    \le C\delta_w(E)^{\frac12} \fullstop
\end{equation}
Finally, notice that thanks to \cref{eq:contains_smaller_ball} and $\hat Q\subseteq B_{\frac12}$, we can replace $\hat Q\cap E$ with $\hat Q$ in \cref{eq:tmp581}.

Hence we can apply \cref{prop:magic_compact_set} and deduce the fundamental bound
\begin{equation*}
    \abs{\pi_{\mathcal E}(x_0)} \le C\delta_w(E)^{\frac12} \comma
\end{equation*}
where $\mathcal E$ is the subspace of directions orthogonal to the constancy directions of $w$.

Thanks to the latter control on $\pi_{\mathcal E}(x_0)$, changing slightly the value of~$x_0$, we can assume that $x_0\in\mathcal C$.

Applying \cref{prop:controlling_constant_directions} with $\xi=x_0$, we know that either \cref{prop:controlling_constant_directions}-\cref{it:linear_growth_ball} or \cref{prop:controlling_constant_directions}-\cref{it:good_slicing} holds.
If \cref{prop:controlling_constant_directions}-\cref{it:linear_growth_ball} holds, since $w(B_1\cap\Sigma) = w(E)$, then \cref{eq:close_to_translated_ball} implies $\abs{x_0} \lesssim C \delta_w(E)^{\frac12}$, that is exactly the desired estimate.
Let us assume that \cref{prop:controlling_constant_directions}-\cref{it:good_slicing} holds. 
For the ease of the reader, let us state again \cref{eq:hessian_control} and \cref{eq:main_proof_inside_nabla}:
\begin{align*}
    &\int_E \abs{\nabla^2\varphi-\id}w \lesssim \delta_w(E)^{\frac12} \comma \\
    &\int_E \abs{\nabla\varphi-(x-x_0)}w(x)\de x \lesssim \delta_w(E)^{\frac12} \fullstop
\end{align*}
From  \cref{eq:contains_smaller_ball} and \cref{prop:controlling_constant_directions}  \cref{it:good_slicing}-\cref{it:locb} and \cref{it:good_slicing}-\cref{it:locf}, we deduce 
\begin{align*}
    &\int_S \de\Haus^{n-1}(s) \int_{\Sigma_s} \abs{\nabla^2\varphi-\id}\de\Haus^1
    \le C\delta_w(E)^{\frac12} \comma \\
    &\int_S \de\Haus^{n-1}(s) \int_{\Sigma_s} \abs{\nabla\varphi - (x-x_0)}\de\Haus^1(x) 
    \le C\delta_w(E)^{\frac12} \fullstop
\end{align*}
Thanks to \cref{it:good_slicing}-\cref{it:locc}, there exists $\overline s\in S$ such that
\begin{align*}
    &\int_{\Sigma_{\overline s}} \abs{\nabla^2\varphi-\id}\de\Haus^1
    \le C\delta_w(E)^{\frac12} \comma \\
    &\int_{\Sigma_{\overline s}} \abs{\nabla\varphi - (x-x_0)}\de\Haus^1(x) 
    \le C\delta_w(E)^{\frac12} \fullstop
\end{align*}
From the latter two inequalities, together with \cref{it:good_slicing}-\cref{it:loca}, it follows that (here it is crucial that $\Sigma_{\overline s}$ is $1$-dimensional)
\begin{equation*}
    \norm{\nabla\varphi - (x-x_0)}_{L^{\infty}(\Sigma_{\overline s})} 
    \le C \delta_w(E)^{\frac12} \fullstop
\end{equation*}
In particular, denoting $z\defeq \overline s + t_0(\overline s)\frac{x_0}{\abs{x_0}}$, it holds
\begin{equation*}
    \abs{\nabla\varphi(z) - (z-x_0)} \le C \delta_w(E)^{\frac12} \fullstop
\end{equation*}
On the other hand, \cref{it:good_slicing}-\cref{it:locd} implies
\begin{equation*}
    \abs{\nabla\varphi(z) - (z-x_0)} \ge
    d\Big(\Sigma, \overline s + (t_0-\abs{x_0})\frac{x_0}{\abs{x_0}}\Big)
    \ge C \abs{x_0} \fullstop
\end{equation*}
The last two estimates together conclude the proof.
\end{proof}

\begin{remark}\label{rem:sharpness_exponent}
    The statement of \cref{thm:main} is sharp with respect to the exponent, i.e., the exponent $\frac12$ present in the right-hand side of \cref{eq:quant-ineq} cannot be increased.
    Let us prove it when $\Sigma$ does not contain lines; the method can be easily adapted to handle the general case.
    
    Given a smooth positive function $r:\partial B_1\cap\Sigma\to\oo0\infty$, let
    \begin{equation*}
        E_{(r)}\defeq \big\{t\theta:\ \theta\in \partial B_1\cap\Sigma,\ 0<t<r(\theta)\big\} \fullstop
    \end{equation*}
    With some standard computations (see also \cref{lem:easy_polar_formulas}) we obtain
    \begin{align*}
        w(E_{(r)}) 
        &= \frac1D\int_{\partial B_1\cap\Sigma}r(\theta)^D\,w(\theta)\de\Haus^{n-1}(\theta) \comma \\
        \Per_w(E_{(r)})
        &= \int_{\partial B_1\cap\Sigma} r(\theta)^{D-1}\sqrt{1+\frac{\abs{\nabla r}^2}{r^2}}\,w(\theta)\de\Haus^{n-1}(\theta) \fullstop
    \end{align*}
    Fix a smooth function $\eta:\partial B_1\cap\Sigma\to\R$ and, for any $\eps>0$, define $E_\eps\defeq E_{(1+\eps\eta)}$.
    If $\int_{\partial B_1\cap\Sigma} \eta\,w\de\Haus^{n-1}=0$, then the previous formulas imply
    \begin{align*}
        w(E_\eps) &= w(B_1\cap\Sigma) + \bigo(\eps^2) \comma \\
        \Per_w(E_\eps) &= \Per_w(B_1\cap\Sigma) + \bigo(\eps^2) \fullstop
    \end{align*}
    Since it holds $A_w(E_\eps)\ge C(\Sigma, w, \eta)\eps$ (with $C>0$ provided $\eta\not\equiv 0$), the family $(E_\eps)_{\eps>0}$ shows the sharpness of the exponent $\frac12$.
\end{remark}

\appendix

\section{Quantitative weighted mean inequality}

The aim of this appendix is to show the following.

\begin{lemma}\label{lem:quantitative_amgm}
    Let $(\lambda_i)_{i=1,\dots, m}$ be positive real numbers with $s\defeq \lambda_1+\cdots+\lambda_m\ge 1$ and let $(x_i)_{i=1,\dots,m}$ be nonnegative real numbers.
    If $\sum \lambda_i x_i \le c s$ for some $c>0$, then it holds
    \begin{equation*}
        \sum_{i=1}^m \lambda_i (x_i-c)^2 \le \frac83\frac{c^{2-s}s^3}{\min\limits_{i=1,\dots,m}\lambda_i^2}\left(c^s-x_1^{\lambda_1}\cdots x_m^{\lambda_m}\right) \fullstop
    \end{equation*}
\end{lemma}
\begin{proof}
    We follow the proof of \cite[Lemma 2.5]{FMP}.
    Without loss of generality we can assume that $c=1$. For any $t>0$ it holds
    \begin{equation*}
        \log(t) \le t - 1 - \frac{(t-1)^2}{2\max(1, t)^2} \fullstop
    \end{equation*}
    Notice that $\max(1, x_i) \le s\lambda^{-1}$, where $\lambda$ is the minimum among $\lambda_1,\dots,\lambda_m$. Therefore it holds
    \begin{align}\begin{split}\label{eq:local_am_gm}
        \log\left(x_1^{\lambda_1}\cdots x_m^{\lambda_m}\right) = 
        \sum_{i=1}^m \lambda_i \log(x_i)
        \le 
        \sum_{i=1}^m \lambda_i\Big(x_i-1-\frac{\lambda^2(x_i-1)^2}{2s^2}\Big)
        \le 
        - \frac{\lambda^2}{2s^2} \sum_{i=1}^m \lambda_i (x_i-1)^2 \fullstop
    \end{split}\end{align}
    Since $s\ge 1$, for any $0\le t \le \frac12$ we have $e^{-st}\le e^{-t}\le 1-\frac34 t$. Moreover, the right-hand side of \cref{eq:local_am_gm} has absolute value below $\frac s2$, thus taking the exponential of both sides we deduce
    \begin{equation*}
        x_1^{\lambda_1}\cdots x_m^{\lambda_m} \le 1-\frac{3\lambda^2}{8s^3}\sum_{i=1}^m\lambda_i(x_i-1)^2 \comma
    \end{equation*}
    that is exactly the desired estimate.
\end{proof}

\section{Concave \texorpdfstring{$1$}{1}-homogeneous functions}

In this appendix we collect some basic facts about concave $1$-homogeneous functions on a cone, as well as a couple of approximation results.

\begin{remark}\label{rem:concave_hom}
    Let $\Sigma\subseteq\R^n$ be a convex cone. Then:
    \begin{itemize}[itemsep=4pt]
        \item If $u, v:\Sigma\to\R$ are $1$-homogeneous concave functions, then so is $\min(u, v)$.
        \item If $T:\R^n\to\R^n$ is a linear isometry and $u:\Sigma\to\R$ is $1$-homogeneous and concave, then so is $u\circ T:T^{-1}(\Sigma)\to\R$.
        \item A function $u:\Sigma\to\R$ is $1$-homogeneous and concave if and only if for every $x\in\Sigma$ there is $\xi_x\in\R^n$ such that $u(x) = \xi_x\cdot x$ and for any $y\in\R^n$ it holds $u(y)\le \xi_x\cdot y$.
    \end{itemize}
\end{remark}

We first prove the following.

\begin{lemma}\label{lem:smoothing_weight}
    Let $\Sigma', \Sigma\subseteq\R^n$ be two open convex cones such that $\Sigma'\cap \partial B_1 \compactsubset \Sigma\cap\partial B_1$. For any concave $1$-homogeneous function $v:\Sigma\to\R$ and any $\eps>0$, there is a concave $1$-homogeneous function $\widetilde v:\Sigma\to\R$ such that $\widetilde v\ge v$ in $\Sigma$, $\widetilde v = v$ on $\partial\Sigma$, $\norm{\widetilde v-v}_{L^\infty(\partial B_1\cap\Sigma')} < \eps$, and $\widetilde v$ is smooth in $\Sigma'$.
\end{lemma}

\begin{proof}
    Fix an open convex cone $\Sigma''\subseteq\R^n$ such that $\Sigma'\cap\partial B_1 \compactsubset \Sigma''\cap\partial B_1 \compactsubset \Sigma\cap\partial B_1$. During the proof we will require some further properties on the cone $\Sigma''$.

    To regularize a concave $1$-homogeneous function we exploit the convolution with respect to the Haar measure (see \cite[2.7]{Federer69}) on $SO(n)$. Let $\mu\in\prob(SO(n))$ be the Haar measure and let $\rho:SO(n)\to\R$ be a smooth kernel, that is $\int \rho\de\mu = 1$, $\rho\ge 0$ and $\rho$ is supported in a small neighborhood of the identity. 
    Let us define
    \begin{equation*}
        v''(x) \defeq \int_{SO(n)} v(T(x)) \rho(T) \de\mu(T) \fullstop
    \end{equation*}
    If $\rho$ is supported in a sufficiently small region, then $v''$ is well-defined in $\Sigma''$ and smooth in it. Moreover it is concave and $1$-homogeneous thanks to \cref{rem:concave_hom}. Choosing appropriately the kernel $\rho$, it is also true that $\norm{v''-v}_{L^\infty(\partial B_1\cap\Sigma'')} < \eps/8$.
    
    As observed in \cref{rem:concave_hom}, for any $x\in\Sigma''$ there is $\xi_x\in\R^n$ such that $v''(x) = \xi_x\cdot x$ and $v''(y)\le \xi_x\cdot y$ for any $y\in\Sigma''$. Let us define
    \begin{equation*}
        v'(y) = \min_{x\in\partial B_1\cap\Sigma'} \xi_x\cdot y \fullstop
    \end{equation*}
    It holds $v'=v''$ in $\Sigma'$ and clearly $v'$ is concave and $1$-homogeneous in $\Sigma$ (but could, a priori, take the value $-\infty$). Let us show that $v'$ is \emph{almost} above $v$ in $\Sigma$. It holds $v'\ge v''$ in $\Sigma''$ and $v''(x)\ge v(x)-\frac{\eps}2\abs{x}$ for any $x\in\Sigma''$, thus $v'(x)\ge v(x) - \frac{\eps}2\abs{x}$ in $\Sigma''$.
    Fix $x\in\partial B_1\cap\Sigma'$ and $z\in\Sigma\setminus\Sigma''$ with $\abs{z}=1$. Up to choosing $\Sigma''$ appropriately, we can assume that there is $0<\lambda<\frac12$ such that $y\defeq \lambda x + (1-\lambda) z\in\Sigma''$. From the properties of $\xi_x$ and the concavity of $v$, it follows
    \begin{align*}
        v(z) &\le \frac{v(y)-\lambda v(x)}{1-\lambda} \le \frac{v''(y) + \frac{\eps}8\abs{y}-\lambda v''(x) + \lambda\frac{\eps}8\abs{x}}{1-\lambda} \le 
        \frac{\xi_x\cdot y + \frac{\eps}8\abs{y} -\lambda\xi_x\cdot x + \lambda\frac{\eps}8\abs{x}}{1-\lambda} \\
        &=
        \xi_x\cdot z + \frac{\eps}8 \frac{\abs{y} + \lambda\abs{x}}{1-\lambda}
        \le \xi_x\cdot z + \frac{\eps}2 \fullstop
    \end{align*}
    Hence, by definition of $v'$, it holds $v(z)\le v'(z) + \frac{\eps}2\abs{z}$ for any $z\in \Sigma\setminus\Sigma''$ and, since we have already established the same inequality in $\Sigma''$, we deduce $v(x)\le v'(x) + \frac{\eps}2\abs{x}$ for any $x\in\Sigma$.
    
    The function $x\mapsto v'(x) + \frac{\eps}2\abs{x}$ satisfies all the requirements of the statement apart from $\widetilde v = v$ on $\partial\Sigma$. To conclude let $\widetilde v:\Sigma\to\R$ be the minimum of all $1$-homogeneous and concave functions $h:\Sigma\to\R$ such that $h\ge v$ in $\Sigma$ and $h\ge v'(x) + \frac{\eps}2\abs{x}$ in $\Sigma'$. With this final step we obtain $\widetilde v = v$ on $\partial\Sigma$ and we do not lose any of the other properties.
\end{proof}

We will also need the following.

\begin{lemma}\label{lem:weight_approximation_zero_trace}
    Let $\Sigma', \Sigma\subseteq\R^n$ be two convex cones such that $\Sigma'\cap \partial B_1 \compactsubset \Sigma\cap\partial B_1$. For any nonnegative concave $1$-homogeneous function $v:\Sigma\to\co0\infty$ and any $\eps>0$, there is a concave $1$-homogeneous function $\widetilde v:\Sigma\to\R$ such that $\widetilde v = 0$ on $\partial\Sigma$ and $\widetilde v = v$ in $\Sigma'$.
\end{lemma}
\begin{proof}
    Let $\widetilde v:\Sigma\to\co0\infty$ be the infimum of all the concave $1$-homogeneous functions $h:\Sigma\to\R$ such that $h\ge 0$ in $\Sigma$ and $h\ge \widetilde v$ in $\Sigma'$. Thanks to the observation of \cref{rem:concave_hom}, the function $\widetilde v$ is concave and $1$-homogeneous. It is straightforward to check that $\widetilde v = v$ in $\Sigma'$ and $\widetilde v = 0$ on the boundary of $\Sigma$.
\end{proof}

For completeness, let us conclude observing that a $1$-homogeneous function $v$ is concave if and only if its restriction on the sphere is \emph{pseudo-concave} (either in the pointwise or differential sense).

\begin{lemma}
    Let $\Sigma\subseteq\R^n$ be a convex cone and let $v:\Sigma\to\R$ be a $1$-homogeneous function. The following statements are equivalent.
    \begin{enumerate}[ref=(\arabic*)]
        \item \label{it:appb_concavity} The function $v$ is concave.
        \item \label{it:appb_pointwise_inequality} For any unit-speed geodesic on the sphere $\gamma:\cc{-s}{t}\to\partial B_1\cap\Sigma$ (with $s,t\ge 0$), it holds
        \begin{equation*}
            v(\gamma(0)) \ge \frac{\sin(t)}{\sin(s+t)}v(\gamma(-s)) + \frac{\sin(s)}{\sin(s+t)}v(\gamma(t)) \fullstop
        \end{equation*}
        \item \label{it:appb_differential_inequality} The function $v$ is twice-differentiable almost everywhere and, for any unit-speed geodesic on the sphere $\gamma:\oo{-\eps}{\eps}\to\partial B_1\cap\Sigma$, if $v$ is twice differentiable at $\gamma(0)$ then it holds
        \begin{equation*}
            \frac{\de^2}{\de t^2}v(\gamma(0)) + v(\gamma(0)) \le 0 \fullstop
        \end{equation*}
    \end{enumerate}
\end{lemma}
\begin{proof}
    Since we do not use this result and the proof is standard, we give only a sketch of the proof.
    
    The statements \cref{it:appb_concavity} and \cref{it:appb_differential_inequality} are equivalent because of the following identity
    \begin{equation*}
        \frac{\de^2}{\de t^2}v(\gamma(t)) = \nabla^2 v(\gamma)[\dot\gamma,\dot\gamma] + \nabla v(\gamma)\cdot\ddot\gamma = 
        \nabla^2 v(\gamma)[\dot\gamma,\dot\gamma] - v(\gamma) \comma
    \end{equation*}
    where in the last step we used that $v$ is $1$-homogeneous.
    
    If $v_1$ and $v_2$ satisfy \cref{it:appb_pointwise_inequality}, then also $\min(v_1, v_2)$ does. Therefore, since linear functions satisfy \cref{it:appb_pointwise_inequality} (with equality), any $1$-homogeneous concave function satisfies \cref{it:appb_pointwise_inequality}. It remains only to prove that \cref{it:appb_pointwise_inequality} implies \cref{it:appb_concavity}. Given two points $p,q\in\Sigma$, let $L:\R^n\to\R$ be a linear map such that $L(p)=v(p)$ and $L(q)=v(q)$.
    Given that $L$ satisfies the equality in \cref{it:appb_pointwise_inequality}, for any $\lambda,\mu\ge 0$ with $\lambda+\mu=1$, it holds
    \begin{equation*}
        \frac{v(\lambda p +\mu q)}{\abs{\lambda p + \mu q}}
        =
        v\Big(\frac{\lambda p +\mu q}{\abs{\lambda p + \mu q}}\Big)
        \ge
        L\Big(\frac{\lambda p +\mu q}{\abs{\lambda p + \mu q}}\Big)
        =
        \frac{\lambda v(p) +\mu v(q)}{\abs{\lambda p + \mu q}}
    \end{equation*}
    which is the sought concavity of $v$.
\end{proof}

\section{Indecomposable sets in \texorpdfstring{$\R^2$}{the plane} are approximated by connected sets}
In this appendix we show that, in $\R^2$, an \emph{indecomposable} set of finite perimeter can be approximated by \emph{connected} smooth open sets (we prove an analogous result also in the weighted setting). Notice that, in higher dimension, any set of finite perimeter can be approximated by smooth connected open sets, while in $\R^2$ this is false.
As a fundamental technical tool, we exploit the theory devised in \cite{ACMM2001}; the interested reader shall refer to that paper for a thorough study of indecomposable sets of finite perimeter.

\begin{definition}\label{def:indecomposable}
    A set of finite perimeter $E\subseteq \R^2$ is indecomposable if it cannot be written as $E=E_1\cup E_2$ with $E_1, E_2$ disjoint nonneglegibile sets of finite perimeter such that
    $\Per(E) = \Per(E_1)+\Per(E_2)$.
\end{definition}

\begin{proposition}\label{prop:indecomposable}
    Let $E\subseteq\R^2$ be an indecomposable set of finite perimeter with $\abs{E}<\infty$. Then, there is a sequence $(\Omega_i)_{i\in\N}$ of bounded connected open sets with smooth boundary such that $\abs{\Omega_i\triangle E}\to 0$ and $\Per(\Omega_i)\to\Per(E)$ as $i\to\infty$.
\end{proposition}
\begin{proof}
    Given a Jordan curve $\gamma:\S^1\to\R^2$, we denote with $\inside(\gamma)$ the bounded connected component of $\R^2\setminus\gamma$.

    Thanks to \cite[Corollary 1]{ACMM2001}, there are $\gamma, (\gamma_i)_{i\in I}$ Jordan curves (with $I$ at most countable) such that (up to negligible sets)
    \begin{equation*}
        E = \inside(\gamma)\setminus \bigcup_{i\in I} \overline{\inside(\gamma_i)} \,, 
    \end{equation*}
    and $\Per(E) = \Haus^1(\gamma) + \sum_{i\in I} \Haus^1(\gamma_i)$. Moreover $\inside(\gamma_i)\subseteq \inside(\gamma)$ for any $i\in I$ and $\inside(\gamma_i)\cap\inside(\gamma_j)=\emptyset$ for any $i\not= j$.
    
    Thus, for any $\eps>0$, we can find a \emph{finite} subset $I'\subseteq I$ such that $\abs{E'\triangle E} < \eps$ and $\abs{\Per(E')-\Per(E)} < \eps$, where 
    \begin{equation*}
        E' \defeq \inside(\gamma)\setminus \bigcup_{i\in I'} \overline{\inside(\gamma_i)} \,.
    \end{equation*}
    Notice that $E'$ is an open connected set of finite perimeter (the connectedness follows from the indecomposability of $E$).
    Since $\eps>0$ can be chosen arbitrarily, we can directly assume that $E$ is open and connected.
    
    Let us now prove the statement for a connected open set $E$ with finite perimeter. 
    For any $k\in\N$, consider the sequence of open sets $(E_k)_{k\in\N}\subseteq\R^n$ defined as
    \begin{equation*}
        E_k\defeq \big\{x\in E:\ d(x, E^\complement) > \frac1k\big\} \,.
    \end{equation*}
    Let $\widetilde E_k$ be the connected component of $E_k$ with the largest measure. Since $E$ is open and $(E_k)_{k\in\N}$ is an increasing chain, it is not hard to check that $\abs{\widetilde E_k\triangle E}\to 0$ as $k\to\infty$.
    Now, let $(\Omega_i)_{i\in\N}$ be a sequence of smooth open sets obtained taking a superlevel set of a convolution of $\chi_E$, more precisely
    \begin{equation*}
        \Omega_i \defeq \Big\{x\in\R^2: \chi_E\ast \big(i^2\eta(i\emptyparam)\big)(x) > t_i\Big\}\,,
    \end{equation*}
    where $\eta:\R^2\to[0, 1]$ is a smooth kernel ($\int\eta=1$) and $(t_i)_{i\in\N}$ is a sequence of values $0 < t_i < 1$ such that the resulting $\Omega_i$ are smooth and $\abs{\Omega_i\triangle E}\to 0$ and $\Per(\Omega_i)\to\Per(E)$.
    This approximation with smooth open sets is very standard, see \cite[Theorem 13.8 ]{maggi2012} for the details.
    
    Given $k\in\N$, it follows from the definition of $\Omega_i$ that, for any sufficiently large $i\in\N$, it holds $E_k\subseteq\Omega_i$. Choose an increasing sequence of indices $(i_k)_{k\in\N}$ such that $E_k\subseteq\Omega_{i_k}$ and
    let $\widetilde\Omega_{i_k}$ be the connected component of $\Omega_{i_k}$ that contains $\widetilde E_k$. Notice that, for any $k\in\N$, it holds
    \begin{equation*}
        \widetilde E_k \subseteq \widetilde\Omega_{i_k} \subseteq \Omega_{i_k} \,,
    \end{equation*}
    hence, since $\abs{\widetilde E_k\triangle E}\to 0$ and $\abs{\Omega_{i_k}\triangle E}\to 0$, we have $\abs{\widetilde \Omega_{i_k}\triangle E}\to 0$.
    Moreover $\Per(\widetilde\Omega_{i_k})\le\Per(\Omega_{i_k})\to\Per(E)$, hence the sequence $(
    \widetilde\Omega_{i_k})_{k\in\N}$ satisfies all the requirements of the statement.
\end{proof}

Let us now give the definition of indecomposable set in the weighted setting and prove a proposition analogous to the latter one.

\begin{definition}\label{def:w_indecomposable}
    Let $\Sigma\subseteq\R^2$ and $w:\Sigma\to[0,\infty)$ be as in \cref{eq:assumptions}.
    A set of finite $w$-perimeter $E\subseteq \Sigma$ is $w$-indecomposable if it cannot be written as $E=E_1\cup E_2$ with $E_1, E_2$ disjoint nonneglegibile sets of finite $w$-perimeter such that
    $\Per_w(E) = \Per_w(E_1)+\Per_w(E_2)$.
\end{definition}
\begin{remark}
    For a set $E\subseteq\Sigma$ such that $\max(\abs{E}, w(E), \Per(E), \Per_w(E))<\infty$, being indecomposable is equivalent to being $w$-indecomposable.
\end{remark}

\begin{proposition}\label{prop:w_indecomposable}
    Let $\Sigma\subseteq\R^2$ and $w:\Sigma\to[0,\infty)$ be as in \cref{eq:assumptions}, with the additional assumption $w\equiv 0$ on $\partial\Sigma$, and let $E\subseteq\Sigma$ be a $w$-indecomposable set of finite $w$-perimeter with $w(E)<\infty$.
    Then, there is a sequence $(\Omega_i)_{i\in\N}$ of bounded connected open sets with smooth boundary such that $\Omega_i\compactsubset\Sigma$ and $w(\Omega_i\setminus E)\to 0$ and $\Per_w(\Omega_i)\to\Per_w(E)$ as $i\to\infty$.
\end{proposition}
\begin{proof}
    As a first step, we prove that $E$ can be approximated by a sequence $(F_k)_{k\in\N}\subseteq\Sigma$ of sets of finite $w$-perimeter such that $F_k\compactsubset\Sigma$. Since the proof is standard and technical we will skip some details.
    
    We can find a sequence of radii $r_k\to\infty$ such that $E_k\defeq E\cap B_{r_k}$ is a set of finite $w$-perimeter and $\Per_w(E_k)\to\Per_w(E)$ as $k\to\infty$.
    Thanks to \cite[Theorem 1]{ACMM2001} (adapting their arguments to our setting is straightforward), we can define $\widetilde E_k$ as the largest $w$-indecomposable component of $E_k$ (i.e., $w(\widetilde E_k)\ge w(C)$ for any $w$-indecomposable component of $E_k$). 
    Since $E_k\nearrow E$ and $E$ is indecomposable, it follows that $w(\widetilde E_k\triangle E)\to 0$ and $\Per_w(\widetilde E_k)\to\Per_w(E)$ as $k\to\infty$.
    Now, fix a vector $v\in\Sigma$ and define $F_k\defeq \eps_k v + \widetilde E_k$, where $\eps_k>0$ is such that $w(F_k\triangle \widetilde E_k)\to 0$ and $\abs{\Per_w(F_k)-\Per_w(\widetilde E_k)}\to 0$ as $k\to\infty$ (here we need the assumption $w\equiv 0$ on $\partial\Sigma$).
    The sets of finite $w$-perimeter $F_k$ are compactly contained in $\Sigma$ and satisfy $w(F_k\triangle E)\to 0$ and $\Per_w(F_k)\to\Per_w(E)$ as $k\to\infty$.
    
    Taking into account the approximation result we have just shown, we can assume, without loss of generality, that $E\compactsubset\Sigma$. To conclude, it is sufficient to repeat the proof of \cref{prop:indecomposable} (notice that the weight is bounded away from $0$ and $\infty$ in an open set $A$ such that $E\compactsubset A\compactsubset \Sigma$).
\end{proof}

\vspace{10ex}

\printbibliography

\end{document}